\documentclass[11pt,a4paper]{amsart}
\usepackage{amssymb,xspace}
\usepackage{amstext}
\usepackage{fourier}
\theoremstyle{plain}
\usepackage{amsbsy,amssymb,amsfonts,latexsym}

\usepackage{enumerate}
\usepackage{graphics}

\date{\today}
\title{Common hypercyclic vectors for high dimensional families of operators}
\author{Fr\'ed\'eric Bayart}
\address{
Clermont Universit\'e, Universit\'e Blaise Pascal, Laboratoire de Math\'ematiques\\
 BP 10448, F-63000 Clermont-Ferrand -\newline
CNRS, UMR 6620, Laboratoire de Math\'ematiques, F-63177 Aubiere
}
\email{Frederic.Bayart@math.univ-bpclermont.fr}



\newcommand{\veps}{\varepsilon}

\def\RR{\mathbb R}

\def\NN{\mathbb N}

\def\TT{\mathbb T}

\def\CC{\mathbb C}

\newcommand{\diam}{\mathrm{diam}}
\def\bd{\mathbb B_d}
\def\hd{\mathbb H_d}
\def\hhd{\mathcal H^2(\hd)}
\def\H{\mathcal H}
\def\bw{\mathbf w}
\newcommand{\ldens}{\underline{\textrm{dens}}}

\newcommand{\mynegspace}{\hspace{-0.12em}}

\newcommand{\lvvvert}{\rvert\mynegspace\rvert\mynegspace\rvert}
\newcommand{\rvvvert}{\rvert\mynegspace\rvert\mynegspace\rvert}

\newtheorem{theorem}{Theorem}[section]

\newtheorem{lemma}[theorem]{Lemma}

\newtheorem{proposition}[theorem]{Proposition}

\newtheorem{corollary}[theorem]{Corollary}

{\theoremstyle{definition}}
{\theoremstyle{definition}}

{\theoremstyle{definition}\newtheorem{example}[theorem]{Example}}

{\theoremstyle{definition}\newtheorem{definition}[theorem]{Definition}}

{\theoremstyle{definition}}

{\theoremstyle{definition}\newtheorem{remark}[theorem]{Remark}}

\newtheorem{question}[theorem]{Question}

\begin{document}

\begin{abstract}
Let $(T_\lambda)_{\lambda\in\Lambda}$ be a family of operators acting on a $F$-space $X$, where the parameter space $\Lambda$ is a subset of $\mathbb R^d$. We give sufficient conditions
on the family to yield the existence of a vector $x\in X$ such that, for any $\lambda\in\Lambda$, the set $\big\{T_\lambda^n x;\ n\geq 1\big\}$ is dense in $X$. We obtain results valid for any value of $d\geq 1$ whereas the previously known results where restricted to $d=1$. Our methods also shed new light on the one-dimensional case.
\end{abstract}

\maketitle

\section{Introduction}
Let $X$ be a separable $F$-space (namely a separable topological vector space which carries a complete translation-invariant metric), and let $T\in\mathcal L(X)$. We say that $T$ is hypercyclic provided there exists a vector $x\in X$ such that its orbit $O(x,T)=\{T^n x;\ n\geq 0\}$ is dense in $X$. The vector $x$ is called a hypercyclic vector for $T$ and the set of hypercyclic vectors for $T$ will be denoted by $HC(T)$. More generally, let $(T_n)$ be a sequence of operators acting on $X$. We say that $x$ is hypercyclic for $(T_n)$ if $\{T_n x;\  n\geq 0\}$ is dense in $X$ and we denote by $HC(T_n)$ the set of hypercyclic vectors for $(T_n)$. 

Hypercyclic operators have been intensively studied in the last few decades (see \cite{BM09} and \cite{GePeBook}). One of the most interesting problem in this field is to find, for a given family of hypercyclic operators, a common hypercyclic vector. It turns out that, as soon as $T$ is hypercyclic, $HC(T)$ is a residual subset of $X$. Hence, for any countable set $\Lambda$, provided each $T_\lambda$, $\lambda\in\Lambda$, is hypercyclic, $\bigcap_{\lambda\in\Lambda}HC(T_\lambda)$ is a residual subset of $X$ and in particular is nonempty.

When $\Lambda$ is uncountable, the situation is more difficult and has attracted the attention of many mathematicians . In the litterature, we may find two kinds of results regarding common hypercyclicity.
\begin{itemize}
\item algebraic results: these results were first obtained by Leon and M\"uller in \cite{LeMu04} when they showed that, for any operator $T\in\mathcal L(X)$ and any $\theta\in\RR$, $HC(e^{i\theta}T)=HC(T)$. This result was extended to $C_0$-semigroup in  \cite{CoMuPe07} by Conejero, M\"uller and Peris: if $(T_t)_{t>0}$ is a strongly continuous group on $X$, then for any $t>0$, $HC(T_t)=HC(T_1)$.
\item analytic results: the pioneering work in that direction is due to Abakumov and Gordon (\cite{AG}) who showed that $\bigcap_{\lambda>1}HC(\lambda B)$ is nonempty, where $B$ is the (unweighted) backward shift on $\ell_2$. This was shortly later improved by Costakis and Sambarino in \cite{CoSa04a} who showed that  $\bigcap_{\lambda>1}HC(\lambda B)$ is residual. Costakis and Sambarino gave a rather general criterion for a family $(T_\lambda)_{\lambda\in I}$ indexed by an interval $I$ to have a residual set of common hypercyclic vectors. This criterion may be applied to many classical sequences of operators, like translation operators $\tau_a$, $a\in\CC\backslash\{0\}$, which are defined on the set of entire functions $H(\CC)$ by $\tau_a(f)=f(\cdot+a)$. More precisely, the criterion shows that $\bigcap_{\theta\in [0,2\pi]}HC(\tau_{e^{i\theta}})$ is nonempty.
\end{itemize}
 
It turns out that both the algebraic results and the analytic results are one-dimensional results. They show that certain families indexed by a subset of $\RR$ have a common hypercyclic vector. Sometimes, we can combine the two methods to obtain two-dimensional results. For instance, by the analytic method, you can show that $\bigcap_{\lambda >1}HC(\lambda B)$ is residual and by the algebraic method, you can show that for any $\theta\in\mathbb R$ and any $\lambda >1$, $HC(e^{i\theta}\lambda B)=HC(\lambda B)$. This yields the following two-dimensional result: $\bigcap_{|\lambda|>1}HC(\lambda B)$ is residual. A similar argument is used to prove that $\bigcap_{a\in\mathbb C\backslash\{0\}}HC(\tau_a)$ is a residual subset of $H(\mathbb C)$.

It was observed by Borichev (see \cite{AG}) that there are dimensional obstructions to the existence of a common hypercyclic vector. Indeed, let $\Lambda\subset (1,+\infty)^2$ and for $\lambda=(s,t)\in\Lambda$, define $T_\lambda=sB\oplus tB$ acting on $\ell_2\oplus\ell_2$. Then each $T_\lambda$ is hypercyclic but if $\bigcap_{\lambda\in\Lambda}HC(T_\lambda)$ is nonempty, then $\Lambda$ has Lebesgue measure zero. See also \cite{Shk10} for other limitations relative to the dimension of the parameter space.

\medskip

However, there are at least two seminal papers where two-dimensional results do appear. The first one is due to Shkarin in \cite{Shk10} who has proved that $\bigcap_{a,b\in\CC^*}HC(b\tau_a)$ is a residual subset of $H(\CC)$. The proof combines a two-dimensional analytic result, precisely $\bigcap_{b>0, a\in S^1}HC(b\tau_a)$ is residual, with two successive applications of the algebraic results. The second one is due to Tsirivas in \cite{Tsi15b} (see also \cite{Tsi15a}). Tsirivas shows that if $(\lambda_n)$ is an increasing sequence of positive real numbers tending to $+\infty$ such that $\lambda_{n+1}/\lambda_n$ goes to 1, then $\bigcap_{a\in \CC\backslash\{0\}}HC(\tau_{\lambda_n a})$ is a residual subset of $H(\CC)$. This is a two-dimensional analytic result, since we cannot apply the algebraic results when $\lambda_n\neq n$.

Both results of Shkarin and Tsirivas are truely "tours de force" which seem specific to the translation operators on $H(\mathbb C)$ or at least to operators very similar to them. In particular, it is not clear if their arguments may be adapted to higher-dimensional families  or to operators acting on a Banach space and not on a Fr\'echet space. In this paper, we provide a new approach which allows us to prove common hypercyclic results for general high-dimensional families. The very simple main idea is the following. The key point in Borichev's example is the fact that if $\lambda^n B^n x$ is close to $y$, then $\mu^n B^n x$ cannot be close to $y$ provided $\mu$ is far away from $\lambda$. Now, if you are working with the group of translations $(\tau_a)$, then $f(x+na)$ and $f(x+nb)$ can be simultaneously close to $g$ even if $b$ is far away from $a$. Indeed, this just mean that $f$ has to be close to $g$ on the balls centered in $-na$ and in $-nb$, and these conditions are in some way independent. This will allow us, in order to construct a common hypercyclic vector $f$, to use the same $n$ for different values of the parameter!

\medskip

Here is our main result.

\begin{theorem}\label{THMMAININTRO}
Let $(T_a)_{a\in\mathbb R^d}$ be a strongly continuous group on $\mathbb R^d$ with the uniform mixing property. Then $\bigcap_{a\in\mathbb R^d\backslash\{0\}}HC(T_a)$ is a residual subset of $X$ 
\end{theorem}

We shall define later the uniform mixing property, but it is a rather natural condition which is satisfied by many operator groups. By applying Theorem \ref{THMMAININTRO}, we will get many new examples of common hypercyclicity which are not reachable with the previously known arguments and for high-dimensional families! 

We shall use two main ingredients in our proof. Firstly we translate the problem of finding a common hypercyclic vector to the problem of finding a suitable covering of compact subsets of $\RR^d$. Secondly we give a way to produce such coverings. It is based on a method to split sequences of real numbers which are going to infinity but not too quickly, and in fact our statement is more general than Theorem \ref{THMMAININTRO} since it covers sequences $(T_{\lambda_n a})$ and not only iterates $(T_{na})$.

\medskip

Even for operator groups like the translation group, there are obstructions to the existence of a common hypercyclic vector for the sequences of operators $(\tau_{\lambda_n a})$, $a\in\CC\backslash\{0\}$, which is linked to the growth of the sequence $(\lambda_n)$. Indeed, Costakis, Tsirivas and Vlachou have shown in \cite{CTV15} that, if $\liminf_{n\to+\infty}\frac{\lambda_{n+1}}{\lambda_n}>2$, then $\bigcap_{a\in \CC\backslash\{0\}}HC(\tau_{na})$ is empty. This shows that, in some sense, the result of Tsirivas quoted above is optimal, but leaves open  the case $\lambda_n=q^n$ with $q\in (1,2]$. Using our covering argument, we are able to extend this result to the remaining case and to any operator group! 
\begin{theorem}\label{THMINTROOBSTRUCTION}
 Let $(T_a)_{a\in\mathbb R^d}$ be a strongly continuous operator group on $X$ and let $(\lambda_n)$ be an increasing sequence of positive real numbers such that $\liminf_n \frac{\lambda_{n+1}}{\lambda_n}>1$. Then $\bigcap_{a\in \RR^d\backslash\{0\}}HC(T_{\lambda_n a})$ is empty.
\end{theorem}

When we add supplementary conditions, the method of Costakis and Sambarino is unefficient to solve certain problems. This is the case if we consider frequent hypercyclicity, a notion introduced in \cite{BAYGRITAMS}. Recall that for a set $A\subset\NN$, its lower density is defined by 
$$\ldens(A)=\liminf_{N\to +\infty}
\frac{
\#\left\{n\leq N;\ n\in A\right\}}{N},$$
where $\#B$ stands for the cardinal number of $B$. Given a sequence of operators $(T_n)$ of $X$, we say that $x\in X$ is a frequently hypercyclic vector for $(T_n)$ if for any $U\subset X$ open and nonempty, the set 
$\{n;\ T_n x\in U\}$ has positive lower density and we denote by $FHC(T_n)$ the set of frequently hypercyclic vectors for $(T_n)$. As before, for a single operator $T$, $FHC(T)$ will stand for $FHC\big((T^n)\big)$.

It was shown in \cite{BAYGRITAMS} that, for any $a\in\CC\backslash\{0\}$, $\tau_a$ acting on $H(\CC)$ is frequently hypercyclic. Moreover, the algebraic method can be carried on frequent hypercyclicity. In particular, if $(T_t)_{t>0}$ is a strongly continuous semigroup on $X$, then for any $t>0$, $FHC(T_t)=FHC(T_1)$. This implies in particular that we can find a common frequently hypercyclic vector for all operators $(\tau_a)_{a\in\mathbb R\backslash\{0\}}$, a result first obtained in \cite{BAYGRITAMS}. 

The methods introduced in this paper allow us to go further and to prove the following natural result.
\begin{theorem}\label{THMFHCTRANSLATION}
The set $\bigcap_{a\in \mathbb C^*}FHC(\tau_a)$ is nonempty.
\end{theorem}
As before, a more general version of Theorem \ref{THMFHCTRANSLATION} will be proved in Section \ref{SECTIONFHC}. In particular, this version can be applied to all the examples introduced in this paper and to high-dimensional families, which is rather surprizing!

\medskip

In the last two sections of this paper, we give related results. First, we study the existence of a common hypercyclic vector for all multiples of operators living in a high-dimensional operator group, leading to a multidimensional generalization of the above result of Shkarin. Second, we emphasize on the algebraic method, showing that it is also helpful to obtain multidimensional results.

\section{The uniform mixing property}
\subsection{The uniform mixing property and a covering argument}\label{SECCOVERING}
In this section, we introduce our main condition for an operator group to admit a common hypercyclic vector. This condition is an enhancement of the mixing property.
\begin{definition}
A group $(T_a)_{a\in\mathbb R^d}$ acting on $X$ has the uniform mixing property if, for any $U,V$ nonempty open subsets of $X$, there exists $C>0$ such that, for any $p\geq 1$, for any $a_1,\dots,a_p\in \mathbb R^d$ with $\|a_i\|\geq C$ and $\|a_i-a_j\|\geq C$ for any $i,j\in\{1,\dots,p\}$ with $i\neq j$, there exists $f\in U$ such that $T_{a_j}f\in V$ for any $j=1,\dots,p$.
\end{definition}
This property is weaker than the Runge property introduced by Shkarin in \cite{Shk10} (see also the forthcoming Section \ref{SECRUNGE}). Moreover, as it is observed in \cite{BoGre07} for a similar property, the Runge property cannot be satisfied for an operator group defined on a Banach space. We shall see later that there exist operator groups defined on Banach spaces and satisfying the uniform mixing property.

Our first result says that the existence of a common hypercyclic vector for an operator group can be deduced from the construction of a suitable covering of $\mathbb R^d$. We shall see later (Theorem \ref{THMCNS}) that the converse is true.
\begin{theorem}\label{THMCOVERING1}
Let $(T_a)_{a\in\mathbb R^d}$ be a strongly continous group on $X$ with the uniform mixing property. Let $K$ be a compact subset of $\mathbb R^d\backslash\{0\}$ and let $(\lambda_n)$ be an increasing sequence of positive real numbers. Assume that, for all $\veps>0$ and all $C>0$, for all $N\in\mathbb N$, we can find $M\geq N$ and a finite number $(x_{n,k})_{N\leq n\leq M,\ 1\leq k\leq p_n}$ of elements of $K$ satisfying
\begin{itemize}
\item[\bf (A)] For any $n,m,k,j$ with $(n,k)\neq (m,j)$, then 
$$\|\lambda_n x_{n,k}-\lambda_m x_{m,j}\|\geq C.$$
\item[\bf (B)] For any $x\in K$, there exist $n,m\in\{N,\dots,M\}$ and $k\in\{1,\dots,p_n\}$ such that 
$$\|\lambda_m x-\lambda_n x_{n,k}\|<\veps.$$
\end{itemize}
Then $\bigcap_{a\in K}HC(T_{\lambda_n a})$ is a residual subset of $X$.
\end{theorem}
\begin{proof}
Let $U,V$ be nonempty open subsets of $X$. It is sufficient to show that
$$U\cap 
\big\{f\in X;\ \forall a\in K,\ \exists n\in\mathbb N,\ T_{\lambda_na}f\in V\big\}$$
is nonempty (see for instance \cite[Proposition 7.4]{BM09}).  
Let $g\in V$ and let $V'$ be a neighbourhood of zero such that $g+V'+V'\subset V$. Since $(T_a)$ is strongly continuous, the uniform boundedness principle says that the map $(a,f)\mapsto T_a f$ is continuous. In particular, there exist $\veps_1>0$ and $W$ a neighbourhood of zero such that $T_a(W)\subset V'$ for any $a$ with $\|a\|<\veps_1$. Moreover, there exists $\veps_2>0$ such that $T_a(g)-g\in V'$ provided $\|a\|<\veps_2$. We set $\veps=\min(\veps_1,\veps_2)$. 

We then set $V_0=g+W$  and we apply the uniform mixing property with $U$ and $V_0$. We get a positive real number $C$ and we choose $N$ such that $\lambda_N \|a\|>C$ for any $a\in K$. For these values of $\veps,C,N$, we get points $(x_{n,k})$ satisfying {\bf (A)} and {\bf (B)}. Applying the uniform mixing property with the sequence $(\lambda_n x_{n,k})$, we know that there exists $f\in U$ such that $T_{\lambda_n x_{n,k}} f\in V_0$ for any admissible choice of $(n,k)$. Pick now $x\in K$. We may find $m,n,k$ such that 
$$\|\lambda_m x-\lambda_n x_{n,k}\|<\veps.$$

Now we write
\begin{eqnarray*}
T_{\lambda_mx}f-g&=&T_{\lambda_mx}f-T_{\lambda_mx-\lambda_n x_{n,k}}g+T_{\lambda_mx-\lambda_n x_{n,k}}g-g\\
&=&\underbrace{T_{\lambda_m x-\lambda_n x_{n,k}}(\underbrace{T_{\lambda_{n,k}x_{n,k}}f-g}_{\in W})}_{\in V'}+\underbrace{T_{\lambda_mx-\lambda_n x_{n,k}}g-g}_{\in V'}.
\end{eqnarray*}
Hence, $T_{\lambda_n x}f$ belongs to $V$, which yields that $\bigcap_{a\in K}HC(T_{\lambda_n a})$ is a residual subset of $X$.
\end{proof}

We shall apply Theorem \ref{THMCOVERING1} under the following form.
\begin{corollary}\label{CORCOVERING2}
Let $(T_a)_{a\in\mathbb R^d}$ be a strongly continous group on $X$ with the uniform mixing property. Let $K$ be a compact subset of $\mathbb R^d\backslash\{0\}$ and let $(\lambda_n)$ be an increasing sequence of positive real numbers. Assume that, for all $\veps>0$ and all $C>0$, there exists $\gamma>0$ such that, for all $N\in\mathbb N$, for all compact subsets $L$ of $K$ with $\diam(L)<\gamma$, we can find $M\geq N$ and a finite number $(x_{n,k})_{N\leq n\leq M,\ 1\leq k\leq p_n}$ of elements of $L$ satisfying
\begin{itemize}
\item[\bf (A)] For any $n,m,k,j$ with $(n,k)\neq (m,j)$, then 
$$\|\lambda_n x_{n,k}-\lambda_m x_{m,j}\|\geq C.$$
\item[\bf (B)] For any $x\in L$, there exist $n,m\in\{N,\dots,M\}$ and $k\in\{1,\dots,p_n\}$ such that 
$$\|\lambda_m x-\lambda_n x_{n,k}\|<\veps.$$
\end{itemize}
Then $\bigcap_{a\in K}HC(T_{\lambda_n a})$ is a residual subset of $X$.
\end{corollary}
\begin{proof}
Let $\veps,C>0$ and $N\in\mathbb N$. We get some $\gamma>0$ and we write $K$ as a finite union of compact subsets $L_1,\dots,L_q$ with $\diam(L_i)<\gamma$. We apply iteratively the construction for each $L_i$ with $N_i$ defined as follows:
\begin{itemize}
\item $N_1=N$;
\item if at Step $i$, we have constructed elements $(x_{n,k})$ until $n\leq M_i$, we take $N_{i+1}>M_i$ any positive integer $n$ such that $\lambda_n \|a\|>\lambda_{M_i}\|b\|+C$ for any $a,b\in K$. This ensures that the separation property {\bf (A)} keeps beings true on the whole set of points $(x_{n,k})$.
\end{itemize}
Hence, the whole sequence $(x_{n,k})$, $n=N_i,\dots,M_i$, $i=1,\dots,q$, satisfies the assumptions {\bf (A)}
and {\bf (B)} of Theorem \ref{THMCOVERING1} that we may apply.
\end{proof}

\subsection{An efficient way to split sequences of positive real numbers}
We need now to introduce a condition on positive sequences of real numbers $(\lambda_n)$ in order to ensure common hypercyclicity of the family $(T_{\lambda_n a})$.  For the translation group on $H(\mathbb C)$, N. Tsirivas has introduced in \cite{Tsi15a} a sufficient condition: it suffices that, for any $M>0$, there exists a subsequence $(\mu_n)$ of $(\lambda_n)$ such that $\mu_{n+1}-\mu_n\geq M$ for any $n$ and $\sum_{n\geq 1}\frac 1{\mu_n}=+\infty$. This last condition was not very surprizing, because it was the main property on the whole sequence of integers $(n)$ which was used in the Costakis-Sambarino criterion. For our purpose, we will weaken this condition in order to allow sequences with faster growth.

\begin{definition}\label{DEFPROPERTYSG}
We say that an increasing sequence $(\lambda_n)$ has property (SG) if, for any $B>0$, there exist $\rho> 1$ and a subsequence $(\mu_n)$ of $(\lambda_n)$ such that
\begin{itemize}
\item $\mu_{n+1}\geq \rho \mu_n$;
\item for any $n_0\in \NN$, $\sum_{n=n_0+1}^{+\infty}\frac{1}{\mu_n}>\frac{B}{\mu_{n_0}}.$
\end{itemize} 
\end{definition}

We first show that many classical sequences have property (SG). 
\begin{proposition}
Let $(\lambda_n)$ be an increasing sequence of positive real numbers tending to $+\infty$ such that $\lambda_{n+1}/\lambda_n\to 1$. Then $(\lambda_n)$ has property (SG).
\end{proposition}
\begin{proof}
Let $B>0$. There exists $\rho_0>1$ such that $\sum_{n=1}^{+\infty}\rho_0^{-n}>B$. We set $\rho=\sqrt{\rho_0}$. Let $p\geq 1$ be such that $\frac{\lambda_{n+1}}{\lambda_n}\leq\rho$ provided $n\geq p$. We then set $\psi(0)=p$ and we define $\psi(n)$ for $n\geq 1$ by induction using the following formula:
$$\psi(n+1)=\inf\big\{m\geq \psi(n);\ \lambda_m\geq\rho\lambda_{\psi(n)}\big\}.$$
Then $\lambda_{\psi(n+1)}\geq \rho\lambda_{\psi(n)}$ and $\lambda_{\psi(n+1)}\leq \rho\lambda_{\psi(n+1)-1}\leq \rho^2\lambda_{\psi(n)}=\rho_0\lambda_{\psi(n)}.$
Setting $\mu_n=\lambda_{\psi(n)}$, we immediately get the conclusion.
\end{proof}

To be able to produce coverings in arbitrary large dimensions, we will need to be able to iterate the property $\sum_{n=n_0+1}^{+\infty}\frac{1}{\mu_n}>\frac{B}{\mu_{n_0}}$ arbitrary many times. The precise statement that we need is contained in the following technical lemma.
\begin{lemma}\label{LEMTECHNICALSEQ}
Let $d\geq 1$ and $A>0$. There exists $B:=B(d,A)>0$ such that, if $(\mu_n)$ is an increasing sequence of positive real numbers such that, for any $n_0\in\mathbb N$, $\sum_{n=n_0+1}^{+\infty}\frac{1}{\mu_n}\geq \frac{B}{\mu_{n_0}}$, if $s\geq 1$ is such that $\sum_{n=1}^s \frac{1}{\mu_n}\geq\frac B{\mu_0}$, then we can find $s_1\in\mathbb N$, subsets $E_r$ of $\NN^{r-1}$ for $r=2,\dots,d+1$, maps $s_r:E_r\to\mathbb N$ for $r=2,\dots,d$ and a one-to-one map $\phi:E_{d+1}\to \{0,\dots,s\}$ such that
\begin{itemize}
\item for any $r=2,\dots,d+1$, 
$$E_r=\big\{(k_1,\dots,k_{r-1});\ k_1<s_1,\ k_2<s_2(k_1),\dots,k_{r-1}\leq s_{r-1}(k_1,\dots,k_{r-2})\big\}.$$
\item for any $r=1,\dots,d$, for any $(k_1,\dots,k_{r-1})\in E_r$,
$$\sum_{j=1}^{s_r(k_1,\dots,k_{r-1})}\frac 1{\mu_{\phi(k_1,\dots,k_{r-1},j,0,\dots,0)}}\geq \frac A{\mu_{\phi(k_1,\dots,k_{r-1},0,\dots,0)}}.$$
\item If $(k_1,\dots,k_d)> (k'_1,\dots,k'_d)$ in the lexicographical order, then
$$\phi(k_1,\dots,k_d)> \phi(k'_1,\dots,k'_d).$$
\end{itemize}
\end{lemma}
\begin{proof}
We proceed by induction on $d$. The result is clear for $d=1$, setting simply $B=A$, $s_1=s$ and $\phi(k)=k$. 
Assume now that the result until step $d-1$ is known and let us prove it for step $d$. Let $A>0$ and let us set
$$B:=B(d,A)=(A+2)B(d-1,A)+3.$$
Let $(\mu_n)$ be an increasing sequence of positive integers and $s\geq 1$ be such that
$$\sum_{n=1}^s \frac 1{\mu_n}\geq\frac B{\mu_0}\textrm{ and }\sum_{n=n_0+1}^{+\infty}\frac 1{\mu_n}\geq\frac B{\mu_{n_0}}\textrm{ for any }n_0\in\NN.$$
We set $\phi(0,\dots,0)=0$ and we define by induction $\phi(j+1,0,\dots,0)$ as the smallest integer $N$ such that
$$\sum_{k=\phi(j,0,\dots,0)+1}^{N}\frac 1{\mu_k}\geq \frac{B(d-1,A)}{\mu_{\phi(j,0,\dots,0)}}.$$
In particular,
\begin{eqnarray}\label{EQTECHNICAL1}
\frac{B(d-1,A)}{\mu_{\phi(j,0,\dots,0)}}\leq \sum_{k=\phi(j,0,\dots,0)+1}^{\phi(j+1,0,\dots,0)}\frac 1{\mu_k}\leq \frac{B(d-1,A)+1}{\mu_{\phi(j,0,\dots,0)}}.
\end{eqnarray}
We stop when $j+1=s_1$ where $s_1$ is the smallest integer $t$ such that
$$\sum_{k=\phi(0,\dots,0)+1}^{\phi(t,0,\dots,0)}\frac 1{\mu_k}\geq\frac{(A+1)B(d-1,A)+2}{\mu_{\phi(0,\dots,0)}}.$$
In particular, 
\begin{eqnarray*}
\sum_{k=\phi(0,\dots,0)+1}^{\phi(s_1,0,\dots,0)}\frac1{\mu_k}&\leq&\frac{(A+1)B(d-1,A)+2}{\mu_{\phi(0,\dots,0)}}+\sum_{k=\phi(s_1-1,0,\dots,0)+1}^{\phi(s_1,0,\dots,0)}\frac1{\mu_k}\\
&\leq&\frac{(A+2)B(d-1,A)+3}{\mu_{\phi(0,\dots,0)}}.
\end{eqnarray*}
We claim that $\phi(s_1,0,\dots,0)\leq s$. Indeed,
$$\sum_{k=\phi(0,\dots,0)+1}^{\phi(s_1,0,\dots,0)}\frac{1}{\mu_k}\leq \frac{(A+2)B(d-1,A)+3}{\mu_{\phi(0,\dots,0)}}\leq\frac{B(d,A)}{\mu_{\phi(0,\dots,0)}}\leq\sum_{k=1}^s \frac{1}{\mu_k}.$$
We now apply the induction hypothesis for $j=0,\dots,s_1-1$ by using the left part of (\ref{EQTECHNICAL1}). We get maps $\phi_j$, $s_{r,j}$. We just set $\phi(j,k_1,\dots,k_d)=\phi_j(k_2,\dots,k_d)$, $s_r(j,k_2,\dots,k_{r-1})=s_{r-1,j}(k_2,\dots,k_{r-1})$. The only thing which remains to be done is to show that
$$\sum_{j=1}^{s_1}\frac 1{\mu_{\phi(j,0,\dots,0)}}\geq \frac{A}{\mu_{\phi(0,\dots,0)}}.$$
This can be done by observing that 
\begin{eqnarray*}
\sum_{j=1}^{s_1}\frac 1{\mu_{\phi(j,0,\dots,0)}}&\geq& \frac 1{B(d-1,A)+1}\sum_{j=1}^{s_1-1}\sum_{k=\phi(j,0,\dots,0)+1}^{\phi(j+1,0,\dots,0)}\frac 1{\mu_k}\\
&\geq& \frac 1{B(d-1,A)+1}\sum_{k=\phi(1,0,\dots,0)+1}^{\phi(s_1,0,\dots,0)}\frac 1{\mu_k}\\
&\geq& \frac 1{B(d-1,A)+1}\left(\sum_{k=1}^{\phi(s_1,0,\dots,0)}\frac 1{\mu_k}-\sum_{k=1}^{\phi(1,0,\dots,0)}\frac 1{\mu_k}\right)\\
&\geq& \frac 1{B(d-1,A)+1}\left(\frac{(A+1)B(d-1,A)+2-B(d-1,A)-1}{\mu_{\phi(0,\dots,0)}}\right)\\
&\geq&\frac{A}{\mu_{\phi(0,\dots,0)}}.
\end{eqnarray*}
\end{proof}

Lemma \ref{LEMTECHNICALSEQ} can be applied for sequences $(\lambda_n)$ having property (SG).
\begin{corollary}\label{CORTECHNICALSEQ}
Let $(\lambda_n)$ be a sequence having property (SG). Then for all $d\geq 1$ and all $A>0$, there exist $\rho>1$ and a subsequence $(\mu_n)$ of $(\lambda_n)$ such that $\mu_{n+1}\geq \rho\mu_n$ for any $n\geq 1$ and, for all $P>0$,  we can find $s_1\in\mathbb N$, subsets $E_r$ of $\NN^{r-1}$ for $r=2,\dots,d+1$, maps $s_r:E_r\to\mathbb N$ for $r=2,\dots,d$ and a one-to-one map $\phi:E_{d+1}\to \mathbb N$ such that
\begin{itemize}
\item for any $r=2,\dots,d+1$, 
$$E_r=\big\{(k_1,\dots,k_{r-1});\ k_1<s_1,\ k_2<s_2(k_1),\dots,k_{r-1}\leq s_{r-1}(k_1,\dots,k_{r-2})\big\}.$$
\item for any $r=1,\dots,d$, for any $(k_1,\dots,k_{r-1})\in E_r$,
$$\sum_{j=1}^{s_r(k_1,\dots,k_{r-1})}\frac 1{\mu_{\phi(k_1,\dots,k_{r-1},j,0,\dots,0)}}\geq \frac A{\mu_{\phi(k_1,\dots,k_{r-1},0,\dots,0)}}.$$
\item $\phi(0,\dots,0)\geq P$.
\item If $(k_1,\dots,k_d)> (k'_1,\dots,k'_d)$ in the lexicographical order, then
$$\phi(k_1,\dots,k_d)> \phi(k'_1,\dots,k'_d).$$
\end{itemize}
\end{corollary}

\subsection{Common hypercyclic vectors}
We are now ready to state and to prove the main result of this section. 
\begin{theorem}\label{THMMAINCOMMON}
Let $(T_a)_{a\in\mathbb R^d}$ be a strongly continuous operator group on $X$ which is uniformly mixing and let $(\lambda_n)$ be an increasing sequence of positive real numbers having property (SG). Then $\bigcap_{a\in\mathbb R^d\backslash\{0\}}HC(T_{\lambda_n a})$ is a residual subset of $X$. 
\end{theorem}
\begin{proof}
We first show that $\bigcap_{a\in K}HC(T_{\lambda_n a})$ is a residual subset of $X$ when $K$ is a compact subset of $(0,+\infty)^d$. We shall prove that the conditions of Corollary \ref{CORCOVERING2} are satisfied. Hence, let $\veps>0$ and $C>0$. We then apply Corollary \ref{CORTECHNICALSEQ} to $A=C/\veps$ to get some $\rho>1$ and some subsequence $(\mu_n)$ of $(\lambda_n)$ with $\mu_{n+1}\geq\rho\mu_n$. Since $K\subset (0,+\infty)^d$, we may find $\gamma>0$ such that, given any $a=(a_1,\dots,a_d)\in K$, $\rho a_i-a_i-\gamma>0$ for all $i=1,\dots,d$. Let now $L$ be a compact subset of $K$ with diameter less than $\gamma$ and let $N\in\mathbb N$. To simplify the notations, we shall assume that $L=\prod_{i=1}^d [b_i,b_i+\gamma]$ with $b_i>0$. We apply the properties of the sequence $(\mu_n)$ given by Corollary \ref{CORTECHNICALSEQ} with $P\geq N$ such that
$$\mu_P\inf_{i=1,\dots,d}(\rho b_i-b_i-\gamma)\geq C.$$
We get maps $s_1,\dots,s_d$ and $\phi$. We may now define our covering of $L$. We set 
$$n_0=\min_{(k_1,\dots,k_d)}\phi(k_1,\dots,k_d),\ m_0=\max_{(k_1,\dots,k_d)}\phi(k_1,\dots,k_d)$$ and let $n\in \{n_0,\dots,m_0\}$. Then either $n$ is not a $\phi(k_1,\dots,k_d)$ and we do nothing. Or $n$ is equal to $\phi(k_1,\dots,k_d)$ for a (necessarily) unique $(k_1,\dots,k_d)$. We then define the set $\{x_{n,k}\}_{1\leq k\leq p_n}$ as 
$$\begin{array}{ll}
L\cap\Bigg\{\Bigg(&\displaystyle b_1+\frac{\alpha_1 C}{\mu_{\phi(0,\dots,0)}}+\frac{\varepsilon }{\mu_{\phi(1,0,\dots,0)}}+\dots+\frac{\varepsilon }{\mu_{\phi(k_1,0,\dots,0)}},
\\[4mm]
&\displaystyle b_2+\frac{\alpha_2 C}{\mu_{\phi(k_1,0,\dots,0)}}+\frac{\varepsilon }{\mu_{\phi(k_1,1,0,\dots,0)}}+\dots+\frac{\varepsilon }{\mu_{\phi(k_1,k_2,\dots,0)}},
\\
&\quad\quad\quad\vdots\\
&\displaystyle b_d+\frac{\alpha_d C}{\mu_{\phi(k_1,\dots,k_{d-1},0)}}+\frac{\varepsilon }{\mu_{\phi(k_1,\dots,k_{d-1},1)}}+\dots+\frac{\varepsilon }{\mu_{\phi(k_1,\dots,k_d)}} \Bigg);\ \alpha_1,\dots,\alpha_d\geq 0\Bigg\}.
\end{array}
$$
We also set $\omega_n=\mu_{\phi(k_1,\dots,k_d)}$ and we claim that {\bf (A)} and {\bf (B)} of Corollary \ref{CORCOVERING2} are satisfied with $\omega_n$ instead of $\lambda_n$ (but $\omega_n$ is of course some $\lambda_m$ and it would be sufficient to renumber everything). Indeed, let $(n,k)\neq (m,j)$. We distinguish two cases:
\begin{itemize}
\item either $n\neq m$, for instance $n<m$. In that case, 
\begin{eqnarray}\label{EQMAINCOMMON}
\|\omega_m x_{m,j}-\omega_n x_{n,k}\|\geq \omega_m b_1-\omega_n (b_1+\gamma)\geq \mu_P (\rho b_1-b_1-\gamma)\geq C.
\end{eqnarray}
\item or $n=m=\phi(k_1,\dots,k_d)$. We write $x_{n,k}$ and $x_{n,j}$ as before, with respectively the sequences $(\alpha_1,\dots,\alpha_d)$ and $(\beta_1,\dots,\beta_d)$. Since $k\neq j$, at least one $\beta_i$ differs from $\alpha_i$. Now, looking at this coordinate, we get
$$\|\omega_n x_{n,k}-\omega_n x_{n,j}\|\geq\frac{\mu_{\phi(k_1,\dots,k_d)}}{\mu_{\phi(k_1,\dots,k_{i-1},0,\dots,0)}}C\geq C.$$
\end{itemize}
Let us now prove {\bf (B)}: let $x\in L$. There exists $\alpha_1>0$ such that 
$$b_1+\frac{\alpha_1C}{\mu_{\phi(0,\dots,0)}}\leq x_1\leq b_1+\frac{(\alpha_1+1)C}{\mu_{\phi(0,\dots,0)}}.$$
Now, by construction of $\phi$, using Corollary \ref{CORTECHNICALSEQ} (recall that $A=C/\veps$), there exists $k_1<s_1$ such that 
\begin{eqnarray*}b_1+\frac{\alpha_1C}{\mu_{\phi(0,\dots,0)}}+\frac{\varepsilon }{\mu_{\phi(1,0,\dots,0)}}+\dots+\frac{\varepsilon }{\mu_{\phi(k_1,0,\dots,0)}}\leq x_1\leq \\
 b_1+\frac{\alpha_1C}{\mu_{\phi(0,\dots,0)}}+\frac{\varepsilon }{\mu_{\phi(1,0,\dots,0)}}+\dots+\frac{\varepsilon }{\mu_{\phi(k_1+1,0,\dots,0)}}.
\end{eqnarray*}
This $k_1$ being fixed, there exists $\alpha_2\geq 0$ such that 
$$b_2+\frac{\alpha_2 C}{\mu_{\phi(k_1,0,\dots,0)}}\leq x_2\leq b_2+\frac{(\alpha_2+1)C}{\mu_{\phi(k_1,0,\dots,0)}}.
$$
Iterating this construction, we find $\alpha_1,\dots,\alpha_d\geq 0$ and $k_1,\dots,k_d$ such that, for all $i=1,\dots,d$, 
\begin{eqnarray*}
b_i+\frac{\alpha_iC}{\mu_{\phi(k_1,\dots,k_{i-1},0,\dots,0)}}+\frac{\varepsilon }{\mu_{\phi(k_1,\dots,k_{i-1},1,0,\dots,0)}}+\dots+\frac{\varepsilon }{\mu_{\phi(k_1,\dots,k_{i-1},k_i,0,\dots,0)}}\leq x_i\leq \\
b_i+\frac{\alpha_iC}{\mu_{\phi(k_1,\dots,k_{i-1},0,\dots,0)}}+\frac{\varepsilon }{\mu_{\phi(k_1,\dots,k_{i-1},1,0,\dots,0)}}+\dots+\frac{\varepsilon }{\mu_{\phi(k_1,\dots,k_{i-1},k_i+1,0,\dots,0)}}.
\end{eqnarray*}
Let $n=\phi(k_1,\dots,k_d)$ and let $x_{n,k}$ corresponding to this value of $\alpha_1,\dots,\alpha_d$. Then, for any $i=1,\dots,d$, 
$$\|\omega_n x-\omega_n x_{n,k}\|\leq\mu_{\phi(k_1,\dots,k_d)}\times \sup_{i=1,\dots, d}\frac{\veps}{\mu_{\phi(k_1,\dots,k_{i}+1,0,\dots,0)}}\leq\veps.$$
Hence, by Corollary \ref{CORCOVERING2}, $\bigcap_{a\in K}HC(T_{\lambda_n a})$ is a residual subset of $X$.
This works for any compact set $K\subset (0,+\infty)^d$ or, more generally, for any compact set $K$ contained in some open orthant of $\mathbb R^d\backslash\{0\}$. Now, assume that $K=\{0\}^{d-e}\times K'$ where $K'$ is a compact set of $(0,+\infty)^{e}$ and define, for $b\in\mathbb R^e$, $S_b=T_{(0,b)}$. Then $(S_b)_{b\in\mathbb R^e}$ has the uniform mixing property and thus 
$$\bigcap_{a\in K}HC(T_{\lambda_n a})=\bigcap_{b\in K'}HC(S_{\lambda_n b})$$
is a residual subset of $X$.

To conclude, let $\mathcal P_e(d)$ be the subsets of $\{1,\dots,d\}$ with cardinal number equal to $e$. For $n\geq 1$, $\veps=\pm 1$ and $S\in\mathcal P_e(d)$, define
$$K_{i,n,\veps}(S)=\left\{
\begin{array}{ll}
\{0\}&\textrm{if }i\notin S\\
\left[\frac{\veps}n,\veps n\right]&\textrm{if }i\in S.
\end{array}\right.$$
Then, writing $\mathbb R^d\backslash\{0\}$ as the countable union $\bigcup_{e=1,\dots,d}\bigcup_{n\geq 1}\bigcup_{\veps\in\{-1,1\}^d}\bigcup_{S\in\mathcal P_e(d)}\prod_{i=1}^d K_{i,n,\veps_i}(S)$, we easily get the conclusion.
\end{proof}

\begin{remark}
In property (SG), the condition $\mu_{n+1}\geq\rho \mu_n$ is rather unpleasant. It is necessary to separate sufficiently $\mu_n x_{n,k}$ and $\mu_m x_{m,k}$ when $n\neq m$. If we just looked at $\bigcap_{\|a\|=1}HC(T_{\lambda_n a})$, we could replace it by the more pleasant condition $\mu_{n+1}-\mu_n\geq B$: see the section devoted to frequent hypercyclicity and in particular compare (\ref{EQMAINCOMMON}) above and (\ref{EQCOVERINGFHC}).
\end{remark}

\section{Examples}
\subsection{A sufficient condition}
We now give examples of operator groups having the uniform mixing property. We first begin by a criterion which can be seen as a strong form of the hypercyclicity criterion.
\begin{proposition}\label{PROPEXUNIFORM}
Let $(T_a)_{a\in\mathbb R^d}$ be an operator group on $X$ and let $\|\cdot\|$ be an $F$-norm on $X$. Assume that there exists a dense set $D\subset X$ such that, for any $f\in D$ and any $\veps>0$, there exists $C>0$ such that, for any $N\geq 1$, for any $a_1,\dots,a_N\in\mathbb R^d$ with $\|a_i-a_j\|\geq C$ if $i\neq j$ and $\|a_i\|\geq C$, then $\left\|\sum_{i=1}^N T_{a_i}f\right\|<\veps$. Then $(T_a)_{a\in\mathbb R^d}$ has the uniform mixing property.
\end{proposition}
Observe that the hypercyclicity criterion shares the same assumptions restricted to $N=1$. For the definition of an $F$-norm, we refer to \cite{GePeBook}.
\begin{proof}
Let $U,V$ be nonempty open subsets of $X$. There exist $g,h\in D,\ \veps>0$ such that $B(g,\veps)\subset U$ and $B(h,\veps)\subset V$. Applying the assumptions with $(g,\veps/2)$ and with $(h,\veps/2)$, we get two positive $C_g$ and $C_h$. We set $C=\max(C_g,C_h)$. Let us consider $N\geq 1$ and $a_1,\dots,a_N\in\mathbb R^d$ with $\|a_i-a_j\|\geq C$ and $\|a_j\|\geq C$. We define
$$f=g-\sum_{i=1}^N T_{-a_i}h.$$
Clearly, $\|f-g\|<\veps/2<\veps$. Moreover, let $j\in\{1,\dots,N\}$. Then
$$T_{a_j}f=T_{a_j}g-\sum_{i\neq j}T_{a_j-a_i}h+h.$$
Setting $b_i=a_j-a_i$, then $\|b_i\|\geq C$ and $\|b_i-b_l\|=\|a_i-a_l\|\geq C$ provided $i\neq l$. Hence, 
$\|T_{a_j}f-h\|<\veps$ and $T_{a_j}f\in V$.
\end{proof}
\begin{example}
Let $w:\mathbb R^d\to\mathbb R$ be a positive bounded and continuous function such that $x\mapsto \frac{w(x+a)}{w(x)}$ is bounded for each $a\in\mathbb R^d$. For $a\in\mathbb R^d$ and $p\geq 1$, let $\tau_a$ be the translation operator defined on $X=L^p(\mathbb R^d,w(x)dx)$ by $\tau_af(x)=f(x+a)$. Assume moreover that $\int_{\mathbb R^d}w(x)dx<+\infty$. Then $\bigcap_{a\in\mathbb R^d\backslash\{0\}}HC(\tau_a)$ is a residual subset of $X$.
\end{example}
\begin{proof}
Let $D\subset X$ be the dense set of compactly supported continuous functions. Let $f\in D$, $\veps>0$ and let $A>0$ be such that the support of $f$ is contained in $B(0,A)$. There exists some $C>0$ such that
\begin{itemize}
\item for any $a,b\in\mathbb R^d$ with $\|a-b\|\geq C$, $T_a f$ and $T_b f$ have disjoint support.
\item $\int_{\|x\|\geq C-A}w(x)dx\leq \frac{\veps^p}{\|f\|_\infty^p}.$
\end{itemize}
Let $N\geq 1$ and $a_1,\dots,a_N\in\mathbb R^d$ with $\|a_i-a_j\|\geq C$ and $\|a_i\|\geq C$ provided $i\neq j$.      Then 
$$\left\|\sum_{i=1}^N T_{a_i}f\right\|^p\leq \int_{\|x\|\geq C-A}\|f\|_\infty^p w(x)dx\leq \veps^p$$
so that $(\tau_a)_{a\in\mathbb R^d}$ has the uniform mixing property.
\end{proof}
This improves (even when $d=1$) Example 7.20 of \cite{BM09}.

\medskip

We can also deduce from Proposition \ref{PROPEXUNIFORM} a useful corollary to get common hypercyclicity.
\begin{corollary}\label{COREXUNIFORM}
Let $(T_a)_{a\in\RR^d}$ be an operator group acting on $X$ and let $\|\cdot\|$ be an $F$-norm on $X$. Assume that there exist a dense set $D\subset X$ and $p>d$ such that, for any $f\in D$, there exists $A>0$ such that 
$$\|T_a f\|\leq \frac A{\|a\|^p}$$
for any $a\in\RR^d$ with $\|a\|\geq 1$. Then $(T_a)_{a\in\RR^d}$ has the uniform mixing property.
\end{corollary}

\begin{proof}
We shall see that the assumptions of Proposition \ref{PROPEXUNIFORM} are satisfied. The key point is to observe that, if $(a_i)$ is any sequence in $\RR^d$ such that $\|a_i-a_j\|\geq 1$ for any $i\neq j$, there exists $\kappa_d>0$ such that, for any $k\geq 1$, 
$$\#\big\{i;\ 2^k\leq \|a_i\|<2^{k+1}\big\}\leq \kappa_d 2^{kd}.$$
Let now $f\in D$, $\veps>0$ and $C\geq 1$. Let $a_1,\dots,a_N\in\RR^d$ with $\|a_i-a_j\|\geq 1$ if $i\neq j$ and $\|a_i\|\geq C$ for any $i$. Then
\begin{eqnarray*}
\left\|\sum_{i=1}^N T_{a_i}f\right\|&\leq&\sum_{k;\ 2^k\geq C}\quad\sum_{i;\ 2^k\leq\|a_i\|< 2^{k+1}}\left\|T_{a_i}f\right\|\\
&\leq&A\sum_{k;\ 2^k\geq C}\quad \sum_{i;\ 2^k\leq \|a_i\|<2^{k+1}}\frac{1}{2^{kp}}\\
&\leq&A\kappa_d\sum_{k;\ 2^k\geq C}\frac{1}{2^{k(p-d)}}
\end{eqnarray*}
and this is less than $\veps$ provided $C$ is large enough.
\end{proof}

\subsection{The Runge property}
\label{SECRUNGE}

Our second example deals with groups having the Runge property introduced in \cite{Shk10}. The name "Runge property" is reminiscent for the method of proof of the hypercyclicity of $\tau_a$ on $H(\mathbb C)$, $a\neq 0$.
\begin{definition}
Let $(T_a)_{a\in\mathbb R^d}$ be  an operator group on $X$. We say that $(T_a)_{a\in\mathbb R^d}$ has the Runge property if for any continuous seminorm $\|\cdot\|$ on $X$, there exists $C>0$ such that, for any $N\geq 1$, for any $b_1,\dots,b_N\in\mathbb R^d$ with $\|b_i-b_j\|\geq C$ for $i\neq j$, for any $g_1,\dots,g_N\in X$, for any $\veps>0$, there exists $f\in X$ such that $\|T_{b_i}f-g_i\|<\veps$ for each $i=1,\dots,N$. 
\end{definition}
\begin{proposition}
An operator group having the Runge property has the uniform mixing property.
\end{proposition}
\begin{proof}
Let $U,V$ be nonempty open subsets of $X$, let $g,h\in X$, let $\|\cdot\|$ be a continuous seminorm on $X$ such that $\{f\in X;\ \|f-h\|<1\}\subset U$ and $\{f\in X;\ \|f-g\|<1\}\subset V$. For this seminorm $\|\cdot\|$ and for $\veps=1$, applying the definition of the Runge property, we find some $C>0$. Let $N\geq 1$ and $a_1,\dots,a_N\in\mathbb R^d$ with $\|a_i-a_j\|\geq C$ and $\|a_i\|\geq C$ for any $i\neq j$. Define $b_0=0$ and $b_i=a_i$ for $i=1,\dots,N$. Then there exists $f\in X$ such that $\|T_{b_0}f-h\|<1$ and $\|T_{b_i}f-g\|<1$ for $i=1,\dots,N$. In particular, $f\in U$ and $T_{a_i}f\in V$ for any $i=1,\dots,N$, so that $(T_a)_{a\in\RR^d}$ has the uniform mixing property.
\end{proof}

In \cite{Shk10}, it is shown that the translation group $(\tau_a)_{a\in\mathbb C}$ acting on $H(\mathbb C)$    has the Runge property. In several complex variables, the situation is less clear due to the lack of Runge theorem. However, this remains true if we restrict ourselves to translations by vectors in $\mathbb R^d$.
\begin{example}
The operator group $(\tau_a)_{a\in\mathbb R^d}$ acting on $ H(\mathbb C^d)$ has the Runge property.
\end{example}
\begin{proof}
By a result of Khudaiberganov \cite{Khu87}, the union of a finite union of disjoint balls in $\mathbb C^d$ with centers in $\mathbb R^d$ is polynomially convex. By the Oka-Weil theorem (see for instance \cite{Ran98}), any function holomorphic in a neighbourhood of a compact polynomially convex set $K\subset\mathbb C^d$ can be uniformly approximated by holomorphic polynomials. Hence, let $\|\cdot\|$ be a continuous seminorm on $H(\mathbb C^d)$. There exists $A>0$ such that, for any $f\in H(\mathbb C^d)$, 
$$\|f\|\leq \|f\|_A:=A\sup_{\|z\|\leq A}|f(z)|.$$
We set $C=4A$ and we consider a finite set of points $b_1,\dots,b_N$ in $\mathbb R^d$ and a finite set of holomorphic functions $g_1,\dots,g_N$ such that $\|b_i-b_j\|\geq C$ provided $i\neq j$. Let $B_i$ be the closed ball $B_i=\{z\in\mathbb C^d;\ \|z-b_i\|\leq A\}$. Since these balls are pairwise disjoint, $K=B_1\cup\dots\cup B_N$ is polynomially convex. Thus, there is a polynomial $f$ such that $\sup_{z\in B_i}|f(z)-g_i(z+b_i)|<\veps/A$ for any $i=1,\dots,N$. We obtain immediately that $\|T_{b_i}f-g_i\|<\veps$ for any $i=1,\dots,N$.
\end{proof}
It is also very easy to show that the translation group $(\tau_a)_{a\in\mathbb R^d}$ acting on the Fr\'echet space $C_c(\mathbb R^d)$ of continous functions $f:\mathbb R^d \to\mathbb R$ with the topology of uniform convergence on compact sets satisfies the Runge property.

\subsection{Heisenberg translations}
We now give an example of a group which is not a translation group. Let $d\geq 2$ and let $X=H^2(\mathbb B_d)$ be the Hardy space on the (euclidean) unit ball of $\mathbb C^d$ denoted by $\mathbb B_d$. Let $\phi$ be an automorphism of $\bd$. Then the composition operator $C_\phi$ defined by $C_\phi(f)=f\circ\phi$ is a bounded operator on $H^2(\bd)$. When $\phi$ has no fixed points in $\bd$, it has be shown in \cite{GKX00} that $C_\phi$ is hypercyclic.

A class of automorphisms plays a crucial role in the study of composition operators and of linear fractional maps of the ball, the class of Heisenberg translations (see \cite{BAYPARBALL} or \cite{BAYCHARP11}).
To understand these automorphisms, it is better to move on the Siegel upper half-space 
$$\hd=\big\{Z=(z,\mathbf w)\in\CC\times\CC^{d-1};\ \Im m(z)>|\mathbf w|^2\big\}.$$
The Siegel half-space is biholomorphic to $\bd$ via the Cayley map $\omega$ defined by 
$$\omega(z,\mathbf w)=\left(i\frac{1+z}{1-z},\frac{i\mathbf w}{1-z}\right),\ \omega^{-1}(z,\mathbf w)=\left(\frac{z-i}{z+i},\frac{2\mathbf w}{z+i}\right).$$
The Cayley transform extends to a homeomorphism of $\overline{\bd}$ onto $\hd\cup\partial\hd\cup\{\infty\}$, the one-point
compactification of $\overline{\hd}$.  The image
of $H^2(\bd)$ by the Cayley transform is denoted by $\hhd$ :
$$\hhd=\left\{F:\hd\to\CC\textrm{ holomorphic};\ F\circ\omega^{-1}\in H^2(\bd)\right\}.$$
As one easily sees by computing the jacobian of $\omega$, $\H_2(\hd)$ is endowed with the norm
$$\|F\|_{\H^2}^2=\kappa^2\int_{\partial \hd}\frac{|F(z,\mathbf w)|^2}{|z+i|^{2d}}d\sigma_{\partial\hd},$$
where $\kappa$ is a constant that we will not try to compute.

For $\gamma\in\mathbb C^{d-1}\backslash\{0\}$, the Heisenberg translation with symbol $\gamma$ is defined by
$$H_\gamma(z,\mathbf w)=(z+2i\langle \mathbf w,\gamma\rangle+i|\gamma|^2,\bw+\gamma).$$
$H_\gamma$ is an automorphism of $\hd$ fixing $\infty$ only and, as already mentioned, $C_{H_\gamma}$ is hypercyclic on $\hhd$. It is also easy to check that $(C_{H_\gamma})_{\gamma\in\CC^{d-1}}$ is a strongly continous group on $\hhd$.
\begin{theorem}\label{THMHEISENBERG}
$\bigcap_{\gamma\in\CC^{d-1}\backslash\{0\}}HC(C_{H_\gamma})$ is a residual subset of $\hhd$.
\end{theorem}
Before to prove this theorem, we need a couple of lemmas.
\begin{lemma}\label{LEMHEISENBERG1}
Let $p\geq 1$. The functions $(z_1-1)^p g(z)$, where $g$ runs over the ball algebra $A(\bd)$, are dense in $H^2(\bd)$.
\end{lemma}
\begin{proof}
Let $h\in A(\bd)$ and define 
\begin{eqnarray*}
f_k(z)&=&h(z)\left(1-\left(\frac{1+z_1}{2}\right)^k\right)^p\\
&=&g_k(z)(1-z_1)^p
\end{eqnarray*}
with $g_k\in A(\bd)$. Lebesgue's theorem implies that $f_k$ tends to $h$ in $H^2(\bd)$. We conclude by density of $A(\bd)$ in $H^2(\bd)$.
\end{proof}
\begin{lemma}\label{LEMHEISENBERG2}
Let $p\geq 1$ and let 
$$D_p=\left\{F\in A(\mathbb H_d);\ \textrm{there exists }C>0\textrm{ s.t. }\forall (z,\bw)\in\hd,\ |F(z,\bw)|\leq \frac{C}{|z+i|^p}\right\}.$$
Then $D_p$ is dense in $\hhd$.
\end{lemma}
\begin{proof}
This follows immediately from the previous lemma and the definition of the Cayley map, since
$$\left|\frac{z-i}{z+i}-1\right|= \frac{2}{|z+i|}.$$
\end{proof}
\begin{lemma}\label{LEMHEISENBERG3}
There exists a dense set $D\subset\hhd$ such that, for any $F\in D$, there exists $A>0$ such that, for any $\gamma\in\CC^{d-1}\backslash\{0\}$, $|\gamma|\geq 1$,
\begin{eqnarray}\label{EQHEISENBERG}
\|C_{H_\gamma}F\|_{\H^2}\leq \frac{A}{|\gamma|^{2d-\frac 32}}.
\end{eqnarray}
\end{lemma}
\begin{proof}
We shall see that, provided $p$ is large enough, $D_p$ satisfies the conclusions of the lemma. Precisely, let $\veps\in(0,1)$ be such that $(2d-1)-\veps(d-1)\geq 2d-\frac 32$. We then adjust $p$ so that $2\veps p\geq 2d-\frac 32$ and we pick $F\in D_p$, $\gamma\in\CC^{d-1}$ with $|\gamma|\geq 1$. To simplify the notations, we shall write during this proof that $u\lesssim v$ provided there exists $C>0$ such that $u\leq Cv$ where $C$ does not depend on $\gamma$ (it may depend on $F$, $p$ or $\veps$). Then
\begin{eqnarray*}
\|C_{H_\gamma}F\|_{\H^2}^2&\lesssim&\int_{\bw \in\CC^{d-1}}\int_{x\in\RR}\frac{|F(x+i|\bw |^2+2i\langle \bw ,\gamma\rangle+i|\gamma|^2,\bw +\gamma)|^2}{|x+i|\bw |^2+i|^{2d}}dxd\bw \\
&\lesssim&\int_{\bw \in\CC^{d-1}}\int_{x\in\RR}\frac{1}{(1+|x|+|\bw |^2)^{2d}(|x-2\Im m\langle \bw ,\gamma\rangle|+|\gamma+\bw |^2+1)^p}dxd\bw \\
&\lesssim&\int_{\bw \in\CC^{d-1}}\int_{x\in\RR}\frac1{(1+|x|+|\bw |^2)^{2d}(1+|\gamma+\bw |^2)^{2p}}dxd\bw .
\end{eqnarray*}
We split the integral following the value of $|\gamma+\bw|$. Assume first that $|\gamma+\bw|\leq \gamma^\veps$. Then, writing
$$\frac1{(1+|x|+|\bw|^2)^{2d}}\times\frac 1{(1+|\gamma+\bw|^2)^{2p}}\lesssim\frac 1{(1+|x|+|\gamma|^2)^{2d}}$$
and observing that the ball $B(-\gamma,\gamma^\veps)$ in $\CC^{d-1}$ has volume comparable to $|\gamma|^{2\veps(d-1)}$, we get
$$\int_{|\bw+\gamma|\leq \gamma^\veps}\int_{x\in\RR}\frac1{(1+|x|+|\bw|^2)^{2d}(1+|\gamma+\bw|^2)^{2p}}dxd\bw
\lesssim\frac{|\gamma|^{2\veps(d-1)}}{(1+|\gamma|^2)^{2d-1}}\lesssim\frac1{|\gamma|^{2\left(2d-\frac 32\right)}}.$$
When $|\gamma+\bw|\geq |\gamma|^\veps$, we now write
$$\frac1{(1+|x|+|\bw|^2)^{2d}}\times\frac 1{(1+|\gamma+\bw|^2)^{2p}}\lesssim\frac1{(1+|x|+|\bw|^2)^{2d}}\times\frac{1}{(1+|\gamma|^{2\veps})^{2p}}.$$
We integrate this over $\CC^{d-1}\times\mathbb R$. Taking into account the inegalities satisfied by $\veps$ and $p$, we get the result of the lemma.
\end{proof}
\begin{proof}[Proof of Theorem \ref{THMHEISENBERG}]
Theorem \ref{THMHEISENBERG} follows immediately from Corollary \ref{COREXUNIFORM} and Lemma \ref{LEMHEISENBERG3}.
\end{proof}

\section{Obstructions to common hypercyclicity}
\subsection{A converse to the covering property}
In this section, we now show that, even if a group $(T_a)_{a\in\RR^d}$ has the uniform mixing property, we cannot expect that $\bigcap_{a\in\RR^d\backslash\{0\}}HC(T_{\lambda_n a})$ is non\-empty provided $(\lambda_n)$ grows too fast. We need a condition on the group saying that it is not too quickly mixing.
\begin{definition}
We say that an operator group $(T_a)_{a\in\RR^d}$ is locally separating if, for any $A>\delta>0$, there exists a non\-empty open set $U$ such that $T_a U\cap U=\varnothing$ provided $\delta\leq\|a\|\leq A$.
\end{definition}
It turns out, that under this condition, the existence of a common hypercyclic vector implies a covering property similar to that of Section \ref{SECCOVERING}. 

\begin{lemma}\label{LEMOBSTRUCTION1}
Let $K$ be a compact subset of $\mathbb R^d\backslash\{0\}$, let $(\lambda_n)_n$ be an increasing sequence of positive real numbers going to infinity and let $(T_a)_{a\in\RR^d}$ be a strongly continuous operator group on $X$ which is locally separating. Assume that $\bigcap_{a\in K}HC(T_{\lambda_n a})$ is nonempty. Then, for all $A>\delta>0$, for all $N\in\mathbb N$, we can find $M\geq N$ and a finite number $(y_{n,k})_{N\leq n\leq M,1\leq k\leq q_n}$ of elements of $K$ satisfying
\begin{itemize}
\item[\bf (A)]for any $n,m,k,j$, either $\|\lambda_n y_{n,k}-\lambda_m y_{m,j}\|<\delta$ 
or $\|\lambda_n y_{n,k}-\lambda_m y_{m,j}\|>A$.
\item[\bf (B)]$K\subset \bigcup_{n=N}^M \bigcup_{k=1}^{q_n}B\left(y_{n,k},\frac\delta{\lambda_n}\right).$
\end{itemize}
\end{lemma}
\begin{proof}
Let $f\in\bigcap_{a\in K}HC(T_{\lambda_n a})$ and let,  for  $n\geq N$, 
$$V_n=\left\{a\in K;\ T_{\lambda_n a}f\in U\right\}$$
where the open set $U$ is given by the local separation property.
Then each $V_n$ is an open subset of $K$ and $K$ is contained in $\bigcup_{n\geq N}V_n$. By the compactness of $K$, there exists $M\geq N$ such that $K\subset\bigcup_{n=N}^M V_n$. Let us now consider $x\in K$ and let $n(x)$ be the smallest integer $n\geq N$ such that $x\in V_n$. Then $K$ is contained in $\bigcup_{x\in K}B\left(x,\frac{\delta}{\lambda_{n(x)}}\right)$. By the compactness of $K$ again, we can extract a finite sequence $(y_{n,k})_{N\leq n\leq M}$ such that each $y_{n,k}$ belongs to $V_n$ and $K$ is contained in $\bigcup_{n,k}B\left(y_{n,k},\frac{\delta}{\lambda_n}\right)$. Moreover, since $T_{\lambda_n y_{n,k}}f\in U$ and $T_{\lambda_m y_{m,j}}f\in U$, it is plain that $T_{\lambda_n y_{n,k}-\lambda_m y_{m,j}}U\cap U\neq\varnothing$, which implies {\bf (A)} by the definition of $U$.
\end{proof}

When $(T_a)_{a\in\RR^d}$ has the uniform mixing property, we can close the circle and show that the converse of Theorem \ref{THMCOVERING1} is true!

\begin{theorem}\label{THMCNS}
Let $(T_a)_{a\in\mathbb R^d}$ be a strongly continous group on $X$ with the uniform mixing property. Let $K$ be a compact subset of $\mathbb R^d\backslash\{0\}$ and let $(\lambda_n)$ be an increasing sequence of positive real numbers. Then the following assertions are equivalent :
\begin{enumerate}
\item $\bigcap_{a\in K}HC(T_{\lambda_n a})$ is a residual subset of $X$.
\item For all $\veps>0$ and all $C>0$, for all $N\in\mathbb N$, we can find $M\geq N$ and a finite number $(x_{n,k})_{N\leq n\leq M,\ 1\leq k\leq p_n}$ of elements of $K$ satisfying
\begin{itemize}
\item[\bf (A)] For any $n,m,k,j$ with $(n,k)\neq (m,j)$, then 
$$\|\lambda_n x_{n,k}-\lambda_m x_{m,j}\|\geq C.$$
\item[\bf (B)] For any $x\in K$, there exist $n,m\in\{N,\dots,M\}$ and $k\in\{1,\dots,p_n\}$ such that 
$$\|\lambda_m x-\lambda_n x_{n,k}\|<\veps.$$
\end{itemize}
\end{enumerate}
\end{theorem}
\begin{proof}
That $(2)$ implies $(1)$ has already been proved in Theorem \ref{THMCOVERING1}. Assume now that (1) is satisfied. We first observe that the uniform mixing property implies the local separation property. Indeed, let $A>\delta>0$. By the uniform mixing property and Theorem \ref{THMMAININTRO}, $\bigcap_{a\in\RR^d\backslash\{0\}}HC(T_{a})$ is nonempty. Pick $f\in \bigcap_{a\in\RR^d\backslash\{0\}}HC(T_{a})$. Then $T_a f\neq f$ for any $a\in\RR^d\backslash\{0\}$; otherwise, the set $\{T_t f;\ t>0\}$ would be compact, hence nondense. By continuity of the map $f\mapsto T_a f$ and by compactness of the corona $\{a\in\RR^d;\ \delta\leq \|a\|\leq A\}$, there exists a neighbourhood $U$ of $f$ such that $T_a U\cap U=\varnothing$ for any $a\in\RR^d$ with $\delta\leq\|a\|\leq A$. Hence, we may apply Lemma \ref{LEMOBSTRUCTION1} and we are lead to show that the conditions of this lemma imply (2). Let $C,\veps>0$ and $N\in\mathbb N$. We apply Lemma \ref{LEMOBSTRUCTION1} with $A=C$ and $\delta=\veps/2$ to get $M$ and $(y_{n,k})$. We order $\NN^2$ using the lexicographical order. We construct by induction sequences $(y'_j)$ and $(\lambda'_j)$ by setting $y'_1=y_{N,1}$, $\lambda'_1=\lambda _N$ and, provided $y'_1,\dots,y'_j,\lambda'_1,\dots,\lambda'_j$ have been constructed with $y'_j=y_{n,k}$ and $\lambda'_j=\lambda_{n}$, we set
$$y'_{j+1}=\inf\big\{(m,\ell)\geq (n,k);\ \forall p\leq j,\ \|\lambda_m y_{m,\ell}-\lambda'_py'_p\|\geq A\big\}$$
and $\lambda'_{j+1}$ is the corresponding $\lambda_m$. If $(m,\ell)$ does not exist, then we stop the construction. We then rename $(y'_j)$ as $(x_{n,k})$: for a given $n$ in $\{N,\dots,M\}$, we set 
$\{x_{n,k}\}=\{y'_j;\ \lambda'_j=\lambda_n\}$.
The construction of the sequence $(x_{n,k})$ immediately implies that $(2)${\bf (A)} is satisfied. Moreover, let $x\in K$. We know that there exists $(n,k)$ with $\|\lambda_n x-\lambda_n y_{n,k}\|<\delta$. By construction, there exists $(m,\ell)$ with $m\leq n$ such that $\|\lambda_n y_{n,k}-\lambda_m x_{m,\ell}\|< A$. But since $x_{m,\ell}$ is itself an element of the sequence $(y_{p,u})$, we have $\|\lambda_n y_{n,k}-\lambda_m x_{m,\ell}\|<\delta$. This implies $\|\lambda_n x-\lambda_m x_{m,\ell}\|<2\delta=\veps$, so that $(2)${\bf(B)} is satisfied.

\end{proof}

\subsection{Obstructions}
We may now state the main theorem of this section, which is the desired extension of the result of Costakis, Tsirivas and Vlachou. We say that an interval in $\RR^d$ is nontrivial if it contains at least two points.
\begin{theorem}\label{THMOBSTRUCTION}
Let $I$ be a nontrival compact interval in $\RR^d\backslash\{0\}$ and let $(T_a)_{a\in\RR^d}$ be a strongly continuous operator group on $X$ which is locally separating. Let also $(\lambda_n)$ be an increasing sequence of positive real numbers such that $\liminf_{n}\lambda_{n+1}/\lambda_n>1$. Then $\bigcap_{a\in I}HC(T_{\lambda_n a})=\varnothing$.
\end{theorem}

The proof of this theorem will depend heavily on the following easy lemma on intervals of $\mathbb R$.
\begin{lemma}\label{LEMINTERVALS}
Let $I,J_1,\dots,J_p$ be intervals of $\RR$. Then $I\backslash(J_1\cup J_2\cup\dots\cup J_p)$ is the reunion of $s$ disjoint intervals $E_1,\dots,E_s$ with $s\leq p+1$ and $|E_1|+\dots+|E_s|\geq |I|-|J_1|-\dots-|J_p|$.
\end{lemma}
\begin{proof}
We proceed by induction on $p$, the result being trivial for $p=0$. Let $E_1,\dots,E_s$, $s\leq p+1$ be disjoint intervals such that $I\backslash(J_1\cup J_2\cup\dots\cup J_p)=E_1\cup\dots\cup E_s$. Renumbering the intervals $E_i$ if necessary, we assume that $\max E_i\leq \min E_{i+1}$. We divide the proof into several cases:
\begin{itemize}
\item if there exists some $j$ such that $J_{p+1}\subset E_{j}$, then set $E_j\backslash J_{p+1}=:E'_j\cup E''_{j}$ where $E'_j$ and $E''_j$ are intervals. In that case,
$$I\backslash(J_1\cup J_2\cup\dots\cup J_p)=E_1\cup\dots\cup E_{j-1}\cup E'_j\cup E''_j\cup E_{j+1}\cup\dots \cup E_s.$$
\item if there exist $j<k$ such that $\min J_{p+1}\in E_j$ and $\max J_{p+1}\in E_k$, then $E_j\backslash J_{p+1}:=E'_j$, $E_k\backslash J_{p+1}:=E'_k$ where $E'_j,E'_k$ are intervals. In that case,
$$I\backslash(J_1\cup J_2\cup\dots\cup J_p)=E_1\cup\dots\cup E_{j-1}\cup E'_j\cup E'_k\cup E_{k+1}\cup\dots \cup E_s.$$
\item if $\min(J_{p+1})$ does not belong to $E_1\cup\dots\cup E_s$ or $\max(J_{p+1})$ does not belong to $E_1\cup\dots\cup E_s$, the proof is similar and even simpler.
\end{itemize}
\end{proof}
We will use Lemma \ref{LEMINTERVALS} under the form of the following corollary.
\begin{corollary}\label{CORINTERVALS}
Let $I,J_1,\dots,J_p$ be intervals of $\RR$ such that $|I|> \sum_{j=1}^p |J_p|$. Then $I\backslash(J_1\cup J_2\cup\dots\cup J_p)$ contains an interval of length at least $\frac 1{p+1}\left(|I|- \sum_{j=1}^p |J_p|\right).$
\end{corollary}
We are now ready for the 
\begin{proof}[Proof of Theorem \ref{THMOBSTRUCTION}]
Let $q>1$ be such that, for any $j\geq 1$, $\frac{\lambda_{j+1}}{\lambda_j}\geq q$. Let us observe that if $(\mu_n)$ is an increasing sequence of positive real numbers such that $\{\lambda_n;\ n\geq 1\}$ is contained in $\{\mu_n;\ n\geq 1\}$, then $\bigcap_{a\in I}HC(T_{\lambda_n a})\subset \bigcap_{a\in I}HC(T_{\mu_n a})$. Then, adding some terms to the sequence $(\lambda_n)$ if necessary, we may always assume that, for any $j\geq 1$, 
$$q\leq \frac{\lambda_{j+1}}{\lambda_j}\leq q^2.$$
We argue by contradiction and we assume that $\bigcap_{a\in I}HC(T_{\lambda_n a})\neq\varnothing$. To simplify the notations, we shall assume that $I\subset \RR$. Let $m\geq 1$ be such that $q^m\geq 2(m+1)$ and let $\delta,A>0$ and $N\in\NN$ be such that
$$\left\{
\begin{array}{l}
\displaystyle \delta\left(1+\frac 1q+\dots+\frac 1{q^{m-1}}\right)<\frac 18\\[4mm]
\displaystyle \frac{A-4\delta}{q^{2(m-1)}}>1\\[4mm]
\displaystyle \frac 1{\lambda_N}\leq |I|.
\end{array}\right.$$
By Lemma \ref{LEMOBSTRUCTION1}, there exist $M\geq N$ and a finite sequence $(y_{n,k})_{N\leq n\leq M,\ 1\leq k\leq q_n}$ such that {\bf (A)} and {\bf (B)} are satisfied. For $n>M$, we set $q_n=0$. Let $n\geq N$ be fixed. We set $u_n=0$ if $q_n=0$. Otherwise, we construct intervals $J_{n,k}$ and $J'_{n,k}$ as follows. We set $k_1=1$ and 
$$J_{n,1}=\left(y_{n,k_1}-\frac{2\delta}{\lambda_n},y_{n,k_1}+\frac{2\delta}{\lambda_n}\right),\ J'_{n,1}=\left(y_{n,k_1}-\frac{\delta}{\lambda_n},y_{n,k_1}+\frac{\delta}{\lambda_n}\right).$$
Let $k_2$ be the first integer $k>k_1$ such that $y_{n,k}\notin J'_{n,1}$. Then by {\bf (A)} $|\lambda_n(y_{n,k_2}-y_{n,k_1})|>A$. We then set $J_{n,2}=\left(y_{n,k_2}-\frac{2\delta}{\lambda_n},y_{n,k_2}+\frac{2\delta}{\lambda_n}\right)$ and $J'_{n,2}=\left(y_{n,k_2}-\frac{\delta}{\lambda_n},y_{n,k_2}+\frac{\delta}{\lambda_n}\right)$. 
We continue this process a finite number of times (until this is impossible). At step $s$, we require that $k_s$ is the smallest integer $k>k_{s-1}$ such that $y_{n,k}\notin J'_{n,j}$, $1\leq j\leq s-1$. We denote by $u_n$ the numbers of intervals constructed in this way.

For any $n\geq N$, we have thus constructed a finite number of intervals $(J_{n,j})_{1\leq j\leq u_n}$ of length $4\delta/\lambda_n$ such that 
$$\textrm{dist}(J_{n,j},J_{n,k})>\frac{A-4\delta}{\lambda_n}\textrm{ provided }j\neq k$$
\begin{eqnarray}\label{EQOBSTRUCTION1}
\bigcup_{k=1}^{q_n} \left(y_{n,k}-\frac\delta{\lambda_n},y_{n,k}+\frac{\delta}{\lambda_n}\right)\subset J_{n,1}\cup\dots\cup J_{n,u_n}.
\end{eqnarray}
We finally set $I_N=I$ and, for $j\geq 0$, $I_{N+j+1}=I_{N+j}\backslash \bigcup_{k=1}^{u_{N+j}}J_{N+j,k}$. By {\bf (B)} and (\ref{EQOBSTRUCTION1}), $I_{M+1}$ is empty. But we will contradict this fact by showing by induction on $j\geq 1$ that $I_{N+mj}$ contains an interval of length at least $\frac 1{\lambda_{N+mj}}$. This property is true for $m=0$ and assume that it is true at rank $j$. Let $J\subset I_{N+mj}$ be an interval of length $\frac 1{\lambda_{N+mj}}$. We claim that, for any $n$ in $\{N+mj,\dots,N+m(j+1)-1\}$, at most one interval $J_{n,k}$ can intersect $J$. Indeed, for $k\neq \ell$, 
$$\textrm{dist}(J_{n,k},J_{n,\ell})>\frac{A-4\delta}{\lambda_n}=\frac{A-4\delta}{\lambda_{N+mj}}\times\frac{\lambda_{N+mj}}{\lambda_n}\geq \frac{A-4\delta}{\lambda_{N+mj}}\times\frac 1{q^{2(m-1)}}> |J|.$$
For $n\in\{N+mj,\dots,N+m(j+1)-1\}$, let $K_n=\varnothing$ if no interval $J_{n,k}$ intersect $J$ and let $K_n=J_{n,k}$ if $J_{n,k}$ is the unique interval $J_{n,\ell}$ intersecting $J$. Then 
$$I_{N+m(j+1)}\supset J\backslash (K_{N+mj}\cup\dots\cup K_{N+m(j+1)-1}).$$
We apply Corollary \ref{CORINTERVALS}: $I_{N+m(j+1)}$ contains an interval of length greater than 
\begin{eqnarray*}
\frac{1}{m+1}\left(|J|-|K_{N+mj}|-\dots-|K_{n+m(j+1)-1}|\right)&\geq&\frac{1}{m+1}\left(
\frac 1{\lambda_{N+mj}}-\frac{4\delta}{\lambda_{N+mj}}-\dots-\frac{4\delta}{\lambda_{N+m(j+1)-1}}\right)\\
&\geq&\frac{1}{(m+1)\lambda_{N+mj}}\left(1-4\delta\left(1+\frac 1q+\dots+\frac1{q^{m-1}}\right)\right)\\
&\geq&\frac{1}{2(m+1)\lambda_{N+mj}}\\
&\geq&\frac{1}{\lambda_{N+m(j+1)}}.
\end{eqnarray*}
This concludes the proof of Theorem \ref{THMOBSTRUCTION}.
\end{proof}
\begin{remark}
Theorem \ref{THMOBSTRUCTION} remains true if we replace the condition $\liminf {\lambda_{n+1}}/{\lambda_n}>1$ by the following one: there exists $p\geq 1$ such that $\liminf {\lambda_{n+p}}/{\lambda_n}>1$.
\end{remark}

Since the uniform mixing property implies the local separation property, we get a kind of duality for an operator group with the uniform mixing property. If the sequence $(\lambda_n)$ does not increase too quickly, then $\bigcap_{a\in \mathbb R^d\backslash\{0\}}HC(T_{\lambda_n a})$ is nonempty. If the sequence $(\lambda_n)$ does increase very quickly, then even for any nontrivial compact interval $I$ in $\RR^d\backslash\{0\}$,  $\bigcap_{a\in I}HC(T_{\lambda_n a})=\varnothing$.

\smallskip

When the compact interval $I$ is ``radial'' (namely, when it is contained in a line passing through $0$), then we can dispense with the local separation property in the statement of Theorem \ref{THMOBSTRUCTION}.
\begin{corollary}
Let $I$ be a nontrivial compact interval in $\mathbb R^d\backslash\{0\}$ which is contained in a line passing through $0$ and let $(T_a)_{a\in\mathbb R^d}$ be any strongly continuous operator group on $X$.  Let also $(\lambda_n)$ be an increasing sequence of positive real numbers such that $\liminf_{n}\lambda_{n+1}/\lambda_n>1$. Then the set $\bigcap_{a\in I}HC(T_{\lambda_n a})$ is empty.
\end{corollary}
Of course, this corollary immediately implies Theorem \ref{THMINTROOBSTRUCTION}.
\begin{proof}
Let $b\in\mathbb R^d\backslash\{0\}$ and $\kappa>1$ be such that $I=[b,\kappa b]$. Then consider the group $(S_t)_{t\in\RR}$ defined by $S_t=T_{\kappa b}$. Then $\bigcap_{\mu\in [1,\kappa]}HC(S_{\lambda_n \mu})=\bigcap_{a\in I}HC(T_{\lambda_n a})$.  Now, a hypercyclic group defined on $\mathbb R$ has automatically the local separation property. Indeed, pick $f\in X$ such that $\{S_t f;\ t\in\mathbb R\}$ is dense in $X$. Then, for any $t\in [\delta,A]\cup[-A,-\delta]$, $S_tf\neq f$; otherwise $\{S_t f;\ t\in\mathbb R\}$ would be compact. It is then easy to find by a compactness argument a neighbourhood $U$ of $f$ such that $S_t U\cap U=\varnothing$ for any $t\in\mathbb R$ with $\delta\leq |t|\leq A$.
\end{proof}

\begin{question}
Does Theorem \ref{THMOBSTRUCTION} remains true if we do not assume that $(T_a)_{a\in\RR^d}$ is locally separating?
\end{question}

\begin{question}
Let $K$ be a compact subset of $(0,+\infty)$ and let $(\lambda_n)=(q^n)$, $q>1$. Assume that $\bigcap_{a\in K}HC(T_{\lambda_n a})$ is nonempty. Can we link $q$ and the Hausdorff dimension of $K$?
\end{question}

\section{Common frequent hypercyclicity}\label{SECTIONFHC}
\subsection{A covering of the unit sphere with control}
In this section, we study the existence of common frequently hypercyclic vectors for operator groups, proving in particular Theorem \ref{THMFHCTRANSLATION}. Our first main argument is a way to divide sequences of integers like in Corollary \ref{CORTECHNICALSEQ}, but now with a control on the growth of the function $\phi$ (in order to obtain {\it frequent} hypercyclicity). We need to introduce a definition.
\begin{definition}
We say that an increasing sequence $(\lambda_n)$ has property (FHCSG) if 
\begin{itemize}
\item for any $n\geq 1$, $\lambda_{n+1}-\lambda_n\geq 1$;
\item for any $C>0$, there exists $p\in\mathbb N$ such that, for any $N\geq 1$, 
$$\sum_{n=N+1}^{N+p}\frac{1}{\lambda_n}\geq\frac C{\lambda_N}.$$
\end{itemize}
\end{definition}
The main difference with property (SG) is that the number of terms of the sum appearing in the last displayed inequality does not depend on $N$.

 When a sequence satisfies property (FHCSG), we have the following improved version of Corollary \ref{CORTECHNICALSEQ}.
\begin{lemma}\label{LEMTECHNICALSEQFHC}
Let $(\lambda_n)$ be a sequence having property (FHCSG). Then for all $d\geq 1$ and all $A>0$, there exists $Q\in\mathbb N$ such that, for any $N\in\NN$,  we can find $s_1\in\mathbb N$, subsets $E_r$ of $\NN^{r-1}$ for $r=2,\dots,d+1$, maps $s_r:E_r\to\mathbb N$ for $r=2,\dots,d$ and a one-to-one map $\phi:E_{d+1}\to \mathbb N$ such that
\begin{itemize}
\item for any $r=2,\dots,d+1$, 
$$E_r=\big\{(k_1,\dots,k_{r-1});\ k_1<s_1,\ k_2<s_2(k_1),\dots,k_{r-1}\leq s_{r-1}(k_1,\dots,k_{r-2})\big\}.$$
\item for any $r=1,\dots,d$, for any $(k_1,\dots,k_{r-1})\in E_r$,
$$\sum_{j=1}^{s_r(k_1,\dots,k_{r-1})}\frac 1{\lambda_{\phi(k_1,\dots,k_{r-1},j,0,\dots,0)}}\geq \frac A{\lambda_{\phi(k_1,\dots,k_{r-1},0,\dots,0)}}.$$
\item For any $(k_1,\dots,k_d)$, $(l_1,\dots,l_d)\in E_{d+1}$ with $(k_1,\dots,k_d)\neq (l_1,\dots,l_d)$, then 
$$|\phi(k_1,\dots,k_d)-\phi(l_1,\dots,l_d)|\geq A$$
$$N\leq \phi(k_1,\dots,k_d)\leq N+Q-1.$$
\end{itemize}
\end{lemma}
\begin{proof}
We first observe that there exists $\rho>1$ such that, for any $n\in\NN$, $\lambda_{n+1}\leq\rho\lambda_n$. Indeed, there exists $p\in\NN$ such that, for any $N\in\NN$, 
$$\sum_{n=N+1}^{N+p} \frac 1{\lambda_n}\geq\frac 1{\lambda_N}.$$
This yields $\frac p{\lambda_{N+1}}\geq\frac 1{\lambda_N}$. Let now $d\geq 1$ and $A>0$. Let $\kappa\in\NN$ with $\kappa>A$ and let $B:=B(d,A)$ be given by Lemma \ref{LEMTECHNICALSEQ}. We apply property (FHCSG) with $C>0$ such that 
$$\frac C\kappa\left(1+\dots+\frac1{\rho^{\kappa-1}}\right)\geq B(d,A).$$
This gives some $p\in\NN$. Let $N\in\NN$ and define $(\mu_n)$ by setting $\mu_n=\lambda_{N+n\kappa}$. Let $s>0$ be such that $s\kappa\geq p$. Then for any $r\geq 0$, 
\begin{eqnarray*}
\sum_{n=r+1}^{r+s}\frac1{\mu_n}&\geq&\frac 1\kappa\left(1+\dots+\frac 1{\rho^{\kappa-1}}\right)\sum_{n=r+1}^{r+s}\left(\frac 1{\lambda_{N+n\kappa}}+\frac{1}{\lambda_{N+n\kappa-1}}+\dots+\frac 1{\lambda_{N+n\kappa-(\kappa-1)}}\right)\\
&\geq&\frac 1\kappa\left(1+\dots+\frac 1{\rho^{\kappa-1}}\right)\sum_{u=1}^{p}\frac{1}{\lambda_{N+r\kappa+u}}\\
&\geq&\frac 1\kappa\left(1+\dots+\frac 1{\rho^{\kappa-1}}\right)\times\frac{C}{\lambda_{N+r\kappa}}\\
&\geq&\frac{B(d,A)}{\mu_r}.
\end{eqnarray*}
Hence, we may apply Lemma \ref{LEMTECHNICALSEQ} to the sequence $(\mu_n)$ and to $s$. 
We get $s_r$, $E_r$ and $\phi$. We finally set $\psi(k_1,\dots,k_d)=N+\kappa\phi(k_1,\dots,k_d)$. It remains to observe that $N\leq \psi(k_1,\dots,k_d)\leq N+\kappa s$ to get that the conclusions of Lemma \ref{LEMTECHNICALSEQFHC} are satisfied, with $\psi$ instead of $\phi$ and $Q=\kappa s+1$.
\end{proof}
We deduce from this lemma a covering lemma for the unit sphere of $\RR^d$ with control of the size of the covering.
\begin{lemma}\label{LEMCOVERINGFHC}
Let $(\lambda_n)$ be an increasing sequence of positive real numbers satisfying property (FHCSG). Then, for any $d\geq 1$, for any $\delta>0$, for any $B>0$, there exists $q>0$  such that, for any $N\in\NN$, there exists a finite number of elements $(x_{n,k})_{n=N,\dots,N+q-1}$ of $S^{d-1}$ such that
\begin{itemize}
\item[\bf(A)] $S^{d-1}\subset\bigcup_{n=N}^{N+q-1}\bigcup_{k}B\left(x_{n,k},\frac{\delta}{\lambda_n}\right)$;
\item[\bf(B)] If $(n,k)\neq (m,j)$, then $\|\lambda_n x_{n,k}-\lambda_m x_{m,j}\|\geq B$.
\end{itemize}
\end{lemma}
The main difference with the coverings used when we applied Theorem \ref{THMCOVERING1} is that now we use at most $q$ values of the sequence $(\lambda_n)$, whereas this size was not controlled before.
\begin{proof}
We first observe that $S^{d-1}$ can be covered by a finite union of sets $K_1,\dots,K_u$ such that, for each $j=1,\dots,u$, there exists a surjective map $\gamma_j:[0,1]^{d-1}\to K_j$ which is bilipschitz:
 $\exists c>0$ such that, for any $y,z\in [0,1]^{d-1}$, 
 $$c^{-1}\|y-z\|\leq \|\gamma_j(y)-\gamma_j(z)\|\leq c \|y-z\|.$$
 Of course, $c$ may be chosen to be independent of $j$. We apply Lemma \ref{LEMTECHNICALSEQFHC} to $d-1$ and to $A=\max\left(\frac{c^2B}\delta,B\right)$ to get some $Q\geq 0$. Let $N\in\NN$ and let us first show how to cover $K_1$. The sets $E_r$, the maps $s_r,\phi$ are defined by Lemma \ref{LEMTECHNICALSEQFHC} and let $n\in\{N,\dots,N+Q-1\}$. Then either $n$ is not equal to some $\phi(k_1,\dots,k_{d-1})$ and we do nothing or $n=\phi(k_1,\dots,k_{d-1})$ for a unique $(k_1,\dots,k_{d-1})$. We then define the set $\{y_{n,k}\}$ as
$$\begin{array}{ll}
[0,1]^{d-1}\cap\Bigg\{\Bigg(&\displaystyle \frac{\alpha_1 cB}{\lambda_{\phi(0,\dots,0)}}+\frac{\delta/c }{\lambda_{\phi(1,0,\dots,0)}}+\dots+\frac{\delta/c }{\lambda_{\phi(k_1,0,\dots,0)}},
\\[4mm]
&\displaystyle \frac{\alpha_2 cB}{\lambda_{\phi(k_1,0,\dots,0)}}+\frac{\delta/c }{\lambda_{\phi(k_1,1,0,\dots,0)}}+\dots+\frac{\delta/c }{\lambda_{\phi(k_1,k_2,\dots,0)}},
\\
&\quad\quad\quad\vdots\\
&\displaystyle \frac{\alpha_{d-1} cB}{\lambda_{\phi(k_1,\dots,k_{d-2},0)}}+\frac{\delta/c }{\lambda_{\phi(k_1,\dots,k_{d-2},1)}}+\dots+\frac{\delta/c }{\lambda_{\phi(k_1,\dots,k_{d-2},k_{d-1})}} \Bigg);\ \alpha_1,\dots,\alpha_{d-1}\geq 0\Bigg\}.
\end{array}
$$
We set, for any $n=N,\dots,N+Q-1$ and any $k$, $x_{n,k}=\gamma_1(y_{n,k})$ and let $x\in K_1$, $x=\gamma_1(y)$. Arguing exactly as in the proof of Theorem \ref{THMMAINCOMMON} and since $A\geq c^2B/\delta$, we find some $(n,k)$ such that
$$\|y-y_{n,k}\|\leq\frac{\delta}{c\lambda_n}$$
which implies 
$$\|x-x_{n,k}\|\leq \frac{\delta}{\lambda_n}.$$
Moreover, consider $x_{n,k}$ and $x_{m,j}$ with $(n,k)\neq (m,\ell )$. Then either $n\neq m$ and by construction of $\phi$, 
\begin{eqnarray}\label{EQCOVERINGFHC}
\|\lambda_n x_{n,k}-\lambda_m x_{m,\ell }\|\geq |\lambda_n-\lambda_m|\geq |n-m|\geq A\geq B.
\end{eqnarray}
Or $n=m$ and in that case, as in the proof of Theorem \ref{THMMAINCOMMON},
$\|\lambda_n y_{n,k}-\lambda_n y_{n,\ell}\|\geq {cB}$  which immediately yields $\|\lambda_n x_{n,k}-\lambda_n x_{n,\ell}\|\geq B.$
Thus, we have produced a good covering of $K_1$. We produce a similar covering of $K_2$ but starting from $N+Q+\kappa$ and thus stopping at $N+2Q-1+\kappa$ where $\kappa\in\NN$ is such that $\kappa\geq B$. More generally, we do the same for each $K_j$, $j=1,\dots,u$, starting at $N+(j-1)(Q+\kappa)$ and stopping at $N+(j-1)(Q+\kappa)+Q-1$. We finally get a net $(x_{n,k})_{N\leq n\leq N+(u-1)(Q+\kappa)+Q-1}$ of $S^{d-1}$ satisfying {\bf (A)} with $q=(u-1)(Q+\kappa)+Q-1$. Moreover, {\bf (B)} is also satisfied since, if $x_{n,k}$ belongs to the covering of $K_j$ and $x_{m,\ell}$ belongs to the covering of $K_{\ell}$ for $\ell\neq k$, then $|n-m|\geq\kappa$ so that
$\|\lambda_n x_{n,k}-\lambda_m x_{m,\ell}\|\geq B.$
\end{proof}
We now combine the previous covering argument with the production of sets with positive lower density.
\begin{lemma}\label{LEMCOVERINGPOSITIVE}
Let $(\lambda_n)$ be an increasing sequence of positive real numbers satisfying property (FHCSG). Let  $(B_p)$ and $(\delta_p)$ be two sequences of positive real numbers. Then there exist a sequence $(q_p)$ of positive integers, a sequence $(\mathbf N_p)$ of subsets of $\NN$ such that 
\begin{enumerate}
\item for any $p\geq 1$, for any $N\in\mathbf N_p$, there exists a finite number $(x_{n,k})_{N\leq n\leq N+q_p-1}$ of elements of $S^{d-1}$ such that
\begin{itemize}
\item $S^{d-1}\subset \bigcup_{n=N}^{N+q_p-1}\bigcup_k B\left(x_{n,k},\frac{\delta_p}{\lambda_n}\right).$
\item if $(n,k)\neq (m,\ell)$, $N\leq n,m\leq N+q_p-1$, $\|\lambda_n x_{n,k}-\lambda_m x_{m,\ell}\|\geq B_p.$
\end{itemize}
\item Each set $\mathbf N_p$ has positive lower density.
\item For any $p,r\geq 1$ and any $(N,M)\in \mathbf N_p\times \mathbf N_r$ with $N\neq M$, 
$$|N-M|\geq (B_p+B_r+q_p+q_r).$$
\item For any $p\geq 1$, $\min(\mathbf N_p)\geq B_p$.
\end{enumerate}
\end{lemma}
\begin{proof}
For a fixed value of $p$ (and thus of $B_p$ and $\delta_p$), let $q_p>0$ be given by Lemma \ref{LEMCOVERINGFHC}. We then apply \cite[Lemma 6.19]{BM09} to the sequence $(N_p)$ defined by $N_p=B_p+q_p$. This provides a sequence $(\mathbf N_p)$ of pairwise disjoint subsets of $\NN$ such that each $\mathbf N_p$ has positive upper density, $\min(\mathbf N_p)\geq N_p$ and $|N-M|\geq N_p+N_r$ whenever $N\neq M$ and $(N,M)\in N_p\times N_r$. Finally, condition (1) follows by applying Lemma \ref{LEMCOVERINGFHC} for any $N\in\mathbf N_p$.
\end{proof}

\subsection{The uniform frequent hypercyclicity criterion}
We now give a criterion for an operator group to have a common frequently hypercyclic vector. It is not very surprizing that this criterion is a strenghtened version of the frequent hypercyclicity criterion.
\begin{definition}
Let $(T_a)_{a\in\RR^d}$ be an operator group on $X$ and let $\|\cdot\|$ be an $F$-norm on $X$. We say that $(T_a)$ satisfies the uniform frequent hypercyclicity criterion if there exists $D\subset X$ dense such that, for any $f \in D$, 
\begin{enumerate}
\item $\sum T_{a_i}f$ converges for any sequence $(a_i)\subset\RR^d$ with $\|a_i-a_j\|\geq 1$ for any $i\neq j$.
\item $\sup\left\|\sum T_{a_i}f\right\|$ tends to zero as $C$ goes to infinity, where the supremum is taken over the sequences $(a_i)\subset\RR^d$ such that $\|a_i-a_j\|\geq 1$ for any $i\neq j$ and $\|a_i\|\geq C$ for any $i$.
\end{enumerate}
\end{definition}
We need a very last lemma on sets with positive lower density.
\begin{lemma}\label{LEMLOWERDENSITY}
Let $E\subset\NN$ with positive lower density and let $q\in\NN$. Let $F\subset  \mathbb N$ be such that, for any $n\in E$, $[n,n+q)\cap F$ is nonempty. Then $F$ has positive lower density.
\end{lemma}
\begin{proof}
Write $E$ as an increasing sequence $(n_k)$. For all $k\geq 1$, there exists at least one element of $F$
in $[n_{kq},n_{(k+1)q})$. Hence
$$\ldens(F)\geq \ldens\big((n_{kq})\big)\geq\frac 1q\ldens(E)>0.$$
\end{proof}
\begin{theorem}\label{THMMAINFHC}
Let $(T_a)_{a\in\RR^d}$ be a strongly continuous operator group on $X$ satisfying the uniform frequent hypercyclicity criterion. Let $(\lambda_n)$ be an increasing sequence of positive integers satisfying property (FHCSG). Then $\bigcap_{a\in S^{d-1}}FHC(T_{\lambda_n a})$ is nonempty.
\end{theorem}
\begin{proof}
Let $(f_p)\subset D$ be a dense sequence in $X$. For any $p\geq 1$, let $\delta_p>0$ be such that $\|T_a f_p-f_p\|<2^{-p}$ provided $\|a\|<\delta_p$. We then consider a sequence $(B_p)_{p\geq 1}$ where $B_p\geq 1$ is such that, for any sequence $(a_i)\subset\RR^d$ with $\|a_i\|\geq B_p-1$ and $\|a_i-a_j\|\geq 1$ whenever $i\neq j$, then $\left\|\sum_i T_{a_i}f_k\right\|<2^{-p}$ for any $k\leq p$.
We then apply Lemma \ref{LEMCOVERINGFHC} to these sequences $(B_p)$ and $(\delta_p)$. For all $p\geq 1$, we define the set $\{a_i(p)\}$ as the set of all the $\lambda_n x_{n,k}$ for $N$ describing $\mathbf N_p$ and $N\leq n\leq N+q_p-1$. Then we may observe that, for any $i\geq 1$, 
$\|a_i(p)\|\geq |\lambda_n|\geq B_p$ and for $i\neq j$, $\|a_i(p)-a_j(p)\|\geq 1$. This comes trivially from the lemma if $a_i(p)=\lambda_n x_{n,k}$ and $a_j(p)=\lambda_m x_{m,\ell}$ with $N\leq n,m\leq N+q_p-1$ and $N\in\mathbf N_p$ (observe that there exists at most one $N$ in $\mathbf N_p$ such that $N\leq n\leq N+q_p-1$). Otherwise,
$$\|a_i(p)-a_j(p)\|\geq |\lambda_n-\lambda_m\geq  |n-m|\geq \inf\{|N-M|;\ M,N\in \mathbf N_p,\ M\neq N\}-q_p\geq B_p\geq 1.$$
The same proof shows that, if $p\neq r$, then for any $i,j$, 
$$\|a_i(p)-a_j(r)\|\geq B_p+B_r.$$
By definition of $B_p$, we know that for each $p\geq 1$, the series $\sum_{i\geq 1}T_{-a_i(p)}f_p$ converges and that $\left\|\sum_{i\geq 1}T_{-a_i(p)}f_p\right\|\leq 2^{-p}$. We finally set
$$f=\sum_{p\geq 1}\sum_{i\geq 1}T_{-a_i(p)}f_p$$
and we claim that $f\in\bigcap_{a\in S^{d-1}}FHC(T_{\lambda_n a})$. Indeed, let $a\in S^{d-1}$, $g\in X$ and let $\veps>0$. There exists $p\geq 1$ such that $\|f_p-g\|<2^{-p}$ and $2^{-p}(p+3)<\veps$. Moreover, for any $N\in\mathbf N_p$, there exists $n\in\{N,\dots,N+q_p-1\}$ such that
$\|\lambda_n a-\lambda_n x_{n,k}\|<\delta_p$. Let $i\geq 1$ be such that $a_i(p)=\lambda_n x_{n,k}$. Then
\begin{eqnarray*}
T_{\lambda_n a}f-g&=&\sum_{r<p}\sum_j T_{\lambda_n a-a_j(r)}f_r+\sum_{j\neq i}T_{\lambda_n a-a_j(p)}f_p+\big(T_{\lambda_n a-a_i(p)}f_p-f_p\big)\\
&&+\big(f_p-g\big)+\sum_{r>p}\sum_j T_{\lambda_n a-a_j(r)}f_r.
\end{eqnarray*}
Let us set, for any $r,j$, $b_j(r)=\lambda_n a -a_j(r)$. Then $\|b_j(r)-b_k(r)\|\geq 1$ for any $r\geq 1$ and any $k\neq j$. Moreover, if $r\neq p$, then 
$$\|b_j(r)\|\geq \|a_i(p)-a_j(r)\|-\|a_i(p)-\lambda_n a\|\geq B_{\max(r,p)}-1$$
and this inequality is also true for $r=p$ and $j\neq i$. By definition of $(B_r)$ and $\delta_p$ we then get
\begin{eqnarray*}
\left\|T_{\lambda_n a}f-g\right\|&\leq&(p-1)2^{-p}+2^{-p}+2^{-p}+2^{-p}+\sum_{r>p}2^{-r}\leq (p+3)2^{-p}<\veps.
\end{eqnarray*}
Lemma \ref{LEMLOWERDENSITY} then achieves the proof that $f$ is a frequently hypercyclic vector for $T_a$.
\end{proof}
Wehre $(\lambda_n)$ is the whole sequence of integers, we can combine this with an algebraic result of \cite{CoMuPe07} to obtain:
\begin{theorem}
Let $(T_a)_{a\in\RR^d}$ be a strongly continuous operator group satisfying the uniform frequent hypercyclicity criterion. Then $\bigcap_{a\in\RR^d\backslash\{0\}}FHC(T_a)$ is nonempty.
\end{theorem}

We now give examples where the previous theorems may be applied.

\begin{corollary}\label{CORFHC}
Let $(T_a)_{a\in\RR^d}$ be an operator group acting on $X$ and let $\|\cdot\|$ be an $F$-norm on $X$. Assume that there exist a dense set $D\subset X$ and $p>d$ such that, for any $f\in D$, there exists $A>0$ such that 
$$\|T_a f\|\leq \frac A{\|a\|^p}$$
for any $a\in\RR^d$ with $\|a\|\geq 1$. Then $\bigcap_{a\in\RR^d\backslash\{0\}}FHC(T_a)$ is nonempty.
\end{corollary}
\begin{proof}
The proof of Corollary \ref{COREXUNIFORM} shows that, under the above assumptions, $(T_a)_{a\in\RR^d}$ satisfies the uniform frequent hypercyclicity criterion.
\end{proof}

\begin{corollary}
Let $w:\mathbb R^d\to\mathbb R$ be a positive bounded and continuous function such that $x\mapsto \frac{w(x+a)}{w(x)}$ is bounded for each $a\in\mathbb R^d$. For $a\in\mathbb R^d$, let $\tau_a$ be the translation operator defined by $\tau_af(x)=f(x+a)$ defined on $X=L^p(\mathbb R^d,w(x)dx)$, $p\geq 1$. Assume moreover that $\int_{\mathbb R^d}w(x)dx<+\infty$. Then $\bigcap_{a\in\mathbb R^d\backslash\{0\}}FHC(\tau_a)$ is nonempty.
\end{corollary}
\begin{proof}
Let $D\subset X$ be the dense set of compactly supported continuous functions. Let $f\in D$ and let $A>0$ be such that the support of $f$ is contained in $B(0,A)$. Let $(a_i)\subset\RR^d$ be such that $\|a_i-a_j\|\geq 1$ for any $i\neq j$. Then it is easy to check that there exists $\kappa>0$ which just depends on the dimension $d$ and of $A$ (and not of the particular choice of the sequence $(a_i)$) such that, for any $x\in\mathbb R^d$, $\|x-a_i\|\leq A$ for at most $\kappa$ different $a_i$. This implies that, for all $x\in\RR^d$, the series $\sum_i f(x+a_i)$ is convergent (the sum is finite) and that 
$\left|\sum_i f(x+a_i)\right|\leq\kappa \|f\|_\infty$. Therefore, if we assume moreover that $\inf_i \|a_i\|\geq C$, then 
$$\left\|\sum_i T_{a_i}f\right\|^p\leq \kappa^p \|f\|_\infty^p \int_{\|x\|\geq C-A}w(x)dx$$
and this goes to zero as $C$ goes to $+\infty$. Hence, the sequence $(\tau_a)_{a\in\RR^d}$ satisfies the uniform frequent hypercyclicity criterion.
\end{proof}
\begin{corollary}
The composition operators on $\H^2(\hd)$ induced by a Heisenberg translation admit a common frequently hypercyclic vector.
\end{corollary}
\begin{proof}
As for Theorem \ref{THMHEISENBERG}, we may apply directly Corollary \ref{CORFHC}.
\end{proof}
\subsection{Uniformly Runge transitive operators groups}
We now turn to the case of the translation group on $H(\CC^d)$. The translation operators $\tau_a$, $a\in\CC^*$, acting on $H(\mathbb C)$, have very special dynamical properties due to the strongness of Runge theorem. An abstract framework to mimic what is useful in Runge theorem has been done in linear dynamics for at least two problems:
\begin{itemize}
\item the problem of common hypercyclicity of the whole operator group, like in \cite{Shk10} or in this paper. The natural generalization here seems the Runge property.
\item the problem of frequent hypercyclicity; this leads the authors of \cite{BoGre07} to introduce the notion of a Runge transitive operator.
\end{itemize}
Since we want a common frequently hypercyclic vector, the right concept seems to be a mixing of these two properties.
\begin{definition}\label{DEFFHCRUNGE}
Let $(T_a)_{a\in\RR^d}$ be an operator group acting on a Fr\'echet space $X$. We say that $(T_a)_{a\in\RR^d}$ is uniformly Runge transitive if there is an increasing sequence $(\|\cdot\|_p)$ of seminorms defining the topology of $X$ and positive integers $A_{p,s}$, $C_p$ for $p,s\in\NN$ such that
\begin{itemize}
\item[(i)]for all $p,s\in\NN$, $f\in X$ and $a\in\RR^d$ with $\|a\|\leq s$,
$$\|T_a f\|_{p}\leq A_{p,s}\|f\|_{s+C_p}.$$
\item[(ii)]for all $p,s\in\NN$ with $s>C_p$, for all $g,h\in X$, $\veps>0$ and for all finite sets $(z_i)\subset\RR^d$ with $\|z_i-z_j\|\geq C_p$ and $\|z_i\|\geq s$ whenever $i\neq j$, there exists $f\in X$ such that
$$\|f-g\|_{s-C_p}<\veps\textrm{ and }\|T_{z_i}f-h\|_p<\veps\textrm{ for all }i.$$
\end{itemize}
\end{definition}
\begin{theorem}\label{THMFHCRUNGE}
Let $(\lambda_n)$ be an increasing sequence of positive real numbers satisfying property (FHCSG) and let $(T_a)_{a\in\RR^d}$ be a strongly continuous operator group which is uniformly Runge transitive.
Then $\bigcap_{a\in S^{d-1}}FHC(T_{\lambda_n a})$ is nonempty. In particular, $\bigcap_{a\in\RR^{d}\backslash\{0\}}FHC(T_a)\neq\varnothing$.
\end{theorem}
\begin{proof}
We first fix a dense sequence $(h_p)$ in $X$ and a sequence $(\veps_p)$ of positive real numbers going to 0. Let $(\|\cdot\|_p)$ be the sequence of seminorms defining the topology of $X$ coming from the definition of the uniform Runge transitivity. We then consider a decreasing sequence $(\delta_p)$ of positive real numbers and a sequence $(M_p)$ of integers such that, for any $p\geq 1$,
\begin{itemize}
\item $\|T_ah_p-h_p\|_p<\veps_p$ provided $\|a\|<\delta_p$;
\item $\|T_a f\|_p<\veps_p$ provided $\|f\|_{M_p}<\delta_p$ and $\|a\|<\delta_p$.
\end{itemize}
Since $(T_a)_{a\in\RR^d}$ is uniformly Runge transitive, there exist sequences $(A_{p,s})$ and $(C_p)$ such that (be careful! We have slightly changed the notations of Definition \ref{DEFFHCRUNGE} to adapt them to our present context. More precisely, $\|\cdot\|_p$ is replaced by $\|\cdot\|_{M_p}$ and $s$ is replaces by $\lambda_s$.)
\begin{itemize}
\item[(i)]for all $p,s\in\NN$, $f\in X$ and $a\in\RR^d$ with $\|a\|\leq\lambda_s$, 
$$\|T_af\|_{M_p}\leq A_{p,s}\|f\|_{\lambda_s+C_p}.$$
\item[(ii)]for all $p,s\in\NN$ with $\lambda_s>C_p$, for all $g,h\in X$, all $\veps>0$, all finite sequences $(z_i)\subset\RR^d$ with $\|z_i-z_j\|\geq C_p$ and $\|z_i\|\geq\lambda_s$ whenever $i\neq j$, there exists $f\in X$ with
$$\|f-g\|_{\lambda_s-C_p}<\veps\textrm{ and }\|T_{z_i}f-h\|_{M_p}<\veps\textrm{ for all }i.$$
\end{itemize}
We apply Lemma \ref{LEMCOVERINGPOSITIVE} with $B_p=C_p+p$ and $\delta_p$ fixed above. We get the sequences $(\mathbf N_p)$ and $(q_p)$ and we write $\bigcup_{p\geq 1}\mathbf N_p$ as an increasing sequence $(n_j)$. For any $j\geq 1$, there is a unique $p_j$ such that $n_j\in \mathbf N_{p_j}$. Moreover, $\lambda_{n_j}\geq n_j\geq B_{p_j}>C_{p_j}$ and $n_j\geq B_{p_j}\geq p_j$. We then define by induction on $j$ a sequence $(f_j)\subset X$ by setting $f_0=0$ and $f_j$ is such that
\begin{eqnarray}\label{EQUNIFORMRUNGE1}
\|f_j-f_{j-1}\|_{\lambda_{n_j}-C_{p_j}}\leq \frac{\eta_{n_j}}{1+\max(A_{p_t,n};\ t\leq j,\ n\leq n_t+q_{p_t})}
\end{eqnarray}
\begin{eqnarray*}
\|T_{\lambda_n x_{n,k}}f_j-h_{p_j}\|_{M_{p_j}}<\eta_{n_j}&&\textrm{for all $n$ in }\{n_j,\dots,n_j+q_{p_j}-1\}\\
&&\textrm{and all possible }k
\end{eqnarray*}
where $(\eta_j)$ is a sequence of positive real numbers such that $\sum_{j\geq k}\eta_j\leq\delta_k$ for any $k$. It is possible to find such an $f_j$ because $\|\lambda_n x_{n,k}-\lambda_m x_{m,\ell}\|\geq C_{p_j}$ if $(n,k)\neq (m,\ell)$ and $\|\lambda_n x_{n,k}\|\geq\lambda_{n_j}$ for all $n=n_j,\dots,n_{j}+q_j-1$ and all $k$. The choice of $B_p$ ensures that $(\lambda_{n_j}-C_{p_j})$ tends to $+\infty$ as $j$ tends to $+\infty$. Therefore, (\ref{EQUNIFORMRUNGE1}) implies that $(f_j)$ converges to some $f\in X$. Let us now fix $j\geq 1$ and $\ell\geq j$. Then, for all $n\in\{n_j,\dots,n_j+q_{p_j}-1\}$ and all possible $k$, 
$$\|T_{\lambda_n x_{n,k}}(f_{\ell+1}-f_\ell)\|_{M_{p_j}}\leq A_{p_j,n_j+q_{p_j}}\|f_{\ell+1}-f_\ell\|_{\lambda_{{n_j}+q_{p_j}}+C_{p_j}}.$$
Now, $\lambda_{n_{\ell+1}}-C_{p_{\ell+1}}\geq\lambda_{n_j}+q_{p_j}+C_{p_j}$ since $\lambda_{n_{\ell+1}}-\lambda_{n_j}\geq n_{\ell+1}-n_j\geq C_{p_{\ell+1}}+C_{p_j}+q_{p_j}$. Hence,
\begin{eqnarray*}
\left\|T_{\lambda_n x_{n,k}}(f_{\ell+1}-f_{\ell})\right\|_{M_{p_j}}&\leq&A_{p_j,n_j+q_{p_j}}\|f_{\ell+1}-f_\ell\|_{\lambda_{n_{\ell+1}}-C_{p_{\ell+1}}}\\
&\leq&\eta_{n_{\ell+1}}\textrm{ by }(\ref{EQUNIFORMRUNGE1}).
\end{eqnarray*}
Summing these inequalities, we have then shown that 
\begin{eqnarray}\label{EQUNIFORMRUNGE2}
\left\|T_{\lambda_n x_{n,k}}f-h_{p_j}\right\|_{M_{p_j}}\leq \sum_{l\geq j}\eta_{n_l}\leq \delta_{n_j}\leq \delta_{p_j}.
\end{eqnarray}
Let us now show that $f\in\bigcap_{a\in S^{d-1}}FHC(T_{\lambda_n a})$. Indeed, let $a\in S^{d-1}$, let $p\in\NN$
  and $N\in \mathbf N_p$. There exist $n\in\{N,\dots,N+q_p-1\}$ and $k$ such that $\|\lambda_n a-\lambda_n x_{n,k}\|<\delta_p.$ Then
\begin{eqnarray*}
\left\|T_{\lambda_n a}f-h_p\right\|_p&\leq&\left\|T_{\lambda_n a-\lambda_n x_{n,k}}\big(T_{\lambda_n x_{n,k}}f-h_p\big)\right\|_p+\left\|T_{\lambda_n a-\lambda_n x_{n,k}}h_p-h_p\right\|_p\\
&\leq&2\veps_p
\end{eqnarray*}
where the last inequality comes from (\ref{EQUNIFORMRUNGE2}) and from the definitions of $\delta_p$ and $M_p$. We now conclude exactly as in the proof of Theorem \ref{THMMAINFHC} that $f\in\bigcap_{a\in S^{d-1}}FHC(T_{\lambda_na})$.
\end{proof}
\begin{example}
The group of translations $(\tau_a)_{a\in\mathbb C}$ acting on $H(\CC)$ is uniformly Runge transitive.
\end{example}
\begin{proof}
Let $\delta\in(0,1)$ and set $\|f\|_n=\sup_{|z|\leq n-\delta}|f(z)|$. We have, for all $p,s\in\NN$, $f\in X$ and $a\in \mathbb C$ with $|a|\leq s$, 
$$\|T_a f\|_p=\sup_{|z-a|\leq p-\delta}|f(z)|\leq \sup_{|z|\leq s+p-\delta}|f(z)|=\|f\|_{s+p}$$
and, if $(z_i)$ is a finite sequence with $\|z_i-z_j\|\geq p$ and $\|z_i\|\geq s$ with $s>p$, the disks $D(z_i,p-\delta)$ are disjoint and disjoint from the disk $D(0,s-p-\delta)$. Hence, Runge's theorem immediately yields the existence of $f\in X$ with 
$$\sup_{|z|<s-p-\delta}|f(z)-g(z)|<\veps\textrm{ and }\sup_{|z-z_i|<p-\delta}|f(z)-h(z-z_i)|<\veps.$$
\end{proof}
In the particular case of the translation group on $H(\CC)$, the proof of Theorem \ref{THMFHCRUNGE} could be slightly simplified by using Arakelian's theorem of approximation of holomorphic functions on closed sets. However, our general theorem may be applied in other contexts, like the translation group on $\mathcal C(\RR^d)$ or on $H(\CC^d)$, if we restrict ourselves to translations by an element of $\RR^d$.

\section{Multiples of a semigroup}
In this section, we show that we can get a common hypercyclic vector if we consider the multiples of an operator group having the Runge property. As usual, we first need a covering lemma.
\begin{lemma}\label{LEMMULT1}
Let $d\geq 1$, $\delta>0$, $B>0$ and $I$ be  a compact interval of $\mathbb R$. Then there exist an integer $r\geq 0$, elements $(y_j)_{j=1,\dots,r}$ of $S^{d-1}$, elements $(\alpha_j)_{j=1,\dots,r}$ of $I$ and integers $(n_j)_{j=1,\dots,r}$ such that 
\begin{itemize}
\item for all $\alpha\in I$ and all $y\in S^{d-1}$, there exists $j\in\{1,\dots,r\}$ with 
$$|\alpha-\alpha_j|<\frac\delta{n_j}\textrm{ and }\|y-y_j\|\leq\frac{\delta}{n_j}.$$
\item for all $j\neq l$ in $\{1,\dots,r\}$, $\|n_j y_j-n_l y_l\|\geq B.$
\item for all $j\in\{1,\dots,r\}$, $n_j\geq B$.
\end{itemize}
\end{lemma}

\begin{proof}
We combine the Costakis-Sambarino method and the methods of the present paper to obtain the right covering. We begin by applying Lemma \ref{LEMCOVERINGFHC} to the whole sequence of integers; we get some $q\in\NN$. Without loss of generality, we may assume that $B\in\NN$. We then define a sequence $(\beta_m)_{m\geq 1}$ by setting $\beta_1=\min(I)$ and $\beta_{m+1}=\beta_{m}+\frac{\delta}{(m+1)(B+q)}$. Let $s\geq 1$ be the biggest integer $m$ such that $\beta_m\leq\max(I)$. For each $m=1,\dots,s$, we then apply Lemma \ref{LEMCOVERINGFHC} with $N=N_m:=m(B+q)$. 
We get elements $(x_{n,k}(m))$ of $S^{d-1}$ for $N_m\leq n\leq N_m+q-1$, $k\leq\omega(n)$. We then rename the $x_{n,k}(m)$ by defining
$$\big \{y_j;\ j=1,\dots,r\big \}:=\big \{x_{n,k}(m);\ 1\leq m\leq s,\ N_m\leq n\leq N_m+q-1,\ k\leq\omega(n)\big \}.$$
For any $j$ in $\{1,\dots,r\}$, there exists a unique $(m,n,k)$ such that $y_j=x_{n,k}(m)$. We then set $n_j=n$ and $\alpha_j=\beta_m$. 

We now verify that the conclusions of Lemma \ref{LEMMULT1} are satisfied. Pick $(\alpha,y)\in I\times S^{d-1}$. There exists $m\in\{1,\dots,s\}$ such that $|\alpha-\beta_m|\leq\frac{\delta}{(m+1)(B+q)}$. This $m$ being fixed, there exist $n\in\{N_m,\dots,N_m+q-1\}$ and $k\leq \omega(n)$ such that $\|y-x_{n,k}(m)\|\leq\frac{\delta}n$. Let $j\in\{1,\dots,r\}$ be such that $y_j=x_{n,k}(m)$ so that $n_j=n$ and $\alpha_j=\beta_m$. Since $n\leq (m+1)(B+q)$, the first part of the conclusions of the lemma is verified.

Suppose now that $j\neq l$ are living in $\{1,\dots,r\}$. If $y_j=x_{k,n}(m)$ and $y_l=x_{n',k'}(m)$ for the same $m$, then the inequality $\|n_jy_j-n_ly_k\|\geq B$ follows directly from
the corresponding inequality of Lemma \ref{LEMCOVERINGFHC}. Otherwise, we simply write
$$\|n_j y_j-n_l y_l\|\geq |n_j-n_l|\geq  B.$$
Finally, for any $j$ in $\{1,\dots,r\}$, $n_j\geq B$ from the very definition of $(N_m)$.
\end{proof}

We need a second lemma related to the continuity of $(\lambda,a)\mapsto \lambda T_a$. It is \cite[Lemma 3.5]{Shk10} where it is formulated for $d=2$, but the proof is unchanged for greater values of $d$. We now assume that $X$ is a Fr\'echet space.
\begin{lemma}\label{LEMMULT2}
Let $(T_a)_{a\in\mathbb R^d}$ be a strongly operator group on $X$, let $g\in X$ and let $\|\cdot\|$ be a continuous seminorm on $X$. Then there exist a continous seminorm $\lvvvert\cdot\rvvvert$ on $X$ and $\delta>0$ such that $\|\cdot\|\leq \lvvvert\cdot\rvvvert$ and, for any $\alpha\in \mathbb R$, $x\in S^{d-1}$, $n\in\mathbb N$ and $f\in X$ satisfying $\lvvvert g-e^{\alpha n}T_{nx}f\rvvvert<1$, we have $\big \|g-e^{\beta n}T_{ny}f\big \|<1$ whenever $\beta\in\mathbb R$ and $y\in S^{d-1}$ are such that $|\alpha-\beta|<\frac{\delta}n$ and $\|y-x\|<\frac{\delta}n$.
\end{lemma}
We can now give our multidimensional analogue of Shkarin's result.
\begin{theorem}\label{THMMULT}
Let $(T_a)_{a\in\mathbb R^d}$ be a strongly operator group on $X$ with the Runge property. Then $\bigcap_{\lambda\in\mathbb C^*,\ a\in\mathbb R^d\backslash\{0\}} HC(\lambda T_a)$ is a residual subset of $X$.
\end{theorem}
\begin{proof}
First of all, as we have already done previously, we apply the algebraic results of Leon and M\"uller and of Conejero, M\"uller and Peris. They give that $HC(\lambda T_a)=HC(\mu T_b)$ provided
$|\lambda|=|\mu|$ and $a=\theta b$ for $\theta>0$. Hence, it is sufficient to show that, for any compact set $I\subset \mathbb R$, the family of operators $\{e^\alpha T_a;\ \alpha\in I,\ a\in S^{d-1}\}$ shares a common hypercyclic vector. For this, we argue as in the proof of Theorem \ref{THMCOVERING1} by picking two nonempty open subsets $U$ and $V$ of $X$ and by showing that
$$U\cap\big\{f\in X;\ \forall \alpha\in I,\ \forall a\in S^{d-1},\ \exists n\in\mathbb N, e^{n\alpha}T_{na}f\in V\big\}\neq\varnothing.$$
Let $g,h\in X$, $\|\cdot\|$ be a continuous seminorm on $X$ such that $\{f\in X;\ \|f-h\|<1\}\subset U$ and $\{f\in X;\ \|f-g\|<1\}\subset V$. Let $\delta>0$ and let
$\lvvvert\cdot\rvvvert$ be a seminorm on $X$ satisfying the conclusions of Lemma \ref{LEMMULT2}. Let also $C>0$ be given by the Runge property for this last seminorm $\lvvvert\cdot\rvvvert$. 
We apply Lemma \ref{LEMMULT1} with $I$, $\delta>0$ and $B=C$ to get finite sequences $(\alpha_j)$, $(y_j)$ and $(n_j)$. By the Runge property, there exists $f\in X$ such that 
\begin{itemize}
\item $\lvvvert f-h\rvvvert<1$;
\item for any $j=1,\dots,r$, $\lvvvert e^{n_j\alpha_j}T_{n_jy_j}f-g\rvvvert<1.$
\end{itemize}
Then $f\in U$. Moreover, for any $(\alpha,a)\in I\times S^{d-1}$, there exists $j$ such that $|\alpha-\alpha_j|<\delta/n_j$ and $\|a-y_j\|<\delta/n_j$. By the choice of $\delta$ and $\lvvvert\cdot\rvvvert$, this yields $\|e^{n_j\alpha}T_{n_j a}f-g\|<1$, which concludes the proof of Theorem \ref{THMMULT}.
\end{proof}
By considering multiples of a semigroup, we cannot go much further; in particular, we cannot get a common frequently hypercyclic vector for the family $\big\{\lambda T_a;\ \lambda>0,\ a\in S^{d-1}\big\}$. In fact, this cannot be the case even for an uncountable family of multiples of a single operator, as the following proposition points out. It should be noticed that it improves \cite[Theorem 4.5]{BAYGRITAMS} where $T$ was equal to $B$ the backward shift. However, the proof remains almost identical.
\begin{proposition}
Let $T$ be a continuous operator acting on the $F$-space $X$ and  let $\Lambda$ be an uncountable subset of $(0,+\infty)$.  Then
the set of common frequently hypercyclic vectors for the family
$(\lambda T)_{\lambda \in\Lambda}$ is empty.
\end{proposition}
\begin{proof}
Let $x^*$ be a nonzero linear functional on $X$ and assume by contradiction that $x$ is a common frequently hypercyclic vector for all operators $\lambda T$, $\lambda\in\Lambda$, 
where $\Lambda$ is uncountable. Let, for any $\lambda\in\Lambda$, 
$$\mathbf N_\lambda=\big\{n\in\mathbb N;\ x^*(\lambda^n T^n x)\in (1/2,3/2)\big\},\ \delta_\lambda=\ldens(\mathbf N_\lambda).$$
For all $\lambda\in\Lambda$, $\delta_\lambda>0$. Since $\Lambda$ is uncountable, this implies that there exist $\lambda\neq \mu$ in $\Lambda$ such that $\mathbf N_\lambda\cap \mathbf N_\mu$ is infinite (see \cite{BAYGRITAMS} for details).
Now, if $n$ belongs to $\mathbf N_\lambda\cap \mathbf N_\mu$, then $ \lambda^n (T^*)^n x^* (x)\in (1/2,3/2)$ and $\mu^n (T^*)^nx^*(x)\in (1/2,3/2)$. This yields
$(\lambda/\mu)^n \in (1/3,3)$, which is a contradiction since $n$ may be chosen as large as we want.
\end{proof}

Our argument in this section is really specific to operators groups having the Runge property. In view of \cite[Corollary 1.10]{Shk10} and of the other results of the present paper, the following question seems natural.
\begin{question}
Let $(T_a)_{a\in\mathbb R^d}$ be a strongly continuous operator group. Assume that there exist $\delta\in(0,1)$ and $D\subset X$ a dense subset of $X$ such that, for any $f\in D$, 
$\|T_a f\|\leq C_f \delta^{\|a\|(d-1)}.$
Does there exist a common hypercyclic vector for the family $\big \{\lambda T_a;\ \delta<|\lambda|<1/\delta,\ a\in\mathbb R^d\backslash\{0\}\big\}$?
\end{question}

\section{A new light on the algebraic method}
\subsection{Yet another criterion}
We now show how the algebraic method can also lead to common hypercyclic results in high dimension.
\begin{theorem}
Let $G$ be a compact topological group, let $(T_{n,g})_{(n,g)\in\NN\times G}$ be a strongly continuous operator semigroup on $X$. Then for any $g\in G$, $HC(T_{1,g})=HC(T_{1,1_G})$.
\end{theorem}
\begin{proof}
Let $u\in HC(T_{1,1_G})$, let $v\in X$ and let $g\in G$. By \cite[Theorem 2.2]{BAYCOS} (which is itself a consequence of results of \cite{Shk08b}), there exists a sequence $(n_k)$ of integers such that 
$$(T_{1,1_G})^{n_k}u\to v\textrm{ and }g^{n_k}\to 1_G.$$
Then
\begin{eqnarray*}
(T_{1,g})^{n_k}u-v&=&T_{n_k,g^{n_k}}u-T_{0,g^{n_k}}v+T_{0,g^{n_k}}v-v\\
&=&T_{0,g^{n_k}}(T_{n_k,1_G}u-v)+T_{0,g^{n_k}}v-v.
\end{eqnarray*}
Now,  since $(T_{n,g})_{(n,g)\in \NN\times G}$ is strongly continuous, the map $(h,w)\in G\times X\mapsto T_{0,h}w$ is continuous by the uniform boundedness principle. We easily deduce that $(T_{1,g})^{n_k}u$ tends to $v$. 

Conversely, let $g\in G$, $u\in HC(T_{1,g})$ and $v\in X$. Let $(n_k)$ be a sequence of integers such that 
$$(T_{1,g})^{n_k}u\to v\textrm{ and }g^{-n_k}\to 1_G.$$
Then we get that $T_{1,1_G}^{n_k}u$ tends to $v$ by writing 
$$T_{1,1_G}^{n_k}u-v=T_{0,g^{-n_k}}(T_{n_k,g^{n_k}}u-v)+T_{0,g^{-n_k}}v-v.$$
\end{proof}
This theorem may be applied with $G=\TT$ the unit circle and $T_{n,\xi}=\xi T^n$ with $T$ any operator on $X$. It gives back the Leon-M\"uller theorem. However, as we have promised, it also leads to interesting multidimensional results.
\subsection{Hyperbolic automorphisms of the ball}
We come back to our discussion on composition operators on the Hardy space of the Siegel upper half-space $\hd$. Let $U\in\CC^{(d-1)\times(d-1)}$ be a unitary matrix, let $\lambda>1$ and define
$$\phi_{\lambda,U}(z,\bw)=(\lambda z,U\bw).$$
The maps $\phi_{\lambda,U}$ are automorphisms of $\hd$ (now hyperbolic automorphisms) and it has been shown in \cite{GKX00}  that, for any $\lambda>1$ and any $U\in \mathbb U(\CC^{d-1})$ (the set of unitary matrices over $\CC^{d-1}$), the composition operator $C_{\phi_\lambda,U}$ is hypercyclic on $\H^2(\hd)$.
\begin{theorem}\label{THMALGEBRAICHIGH}
The set $\bigcap_{\lambda>1,\ U\in\mathbb U(\CC^{d-1})} HC(C_{\phi_{\lambda,U}})$ is a residual subset of $\H^2(\hd)$. 
\end{theorem}
\begin{proof}
The proof is divided into two parts, the multidimensional part and the one-dimensional part. For the multidimensional part, we fix $\lambda>1$, we set $G=\mathbb U(\CC^{d-1})$ and we write $T_{n,U}=C_{\phi_{\lambda^n,U}}$. This is clearly a strongly continuous semigroup on $\H^2(\hd)$. Hence we may apply Theorem \ref{THMALGEBRAICHIGH} and we know that $HC(C_{\phi_{\lambda,U}})=HC(C_{\phi_{\lambda,I}})$ for any $U\in\mathbb U(\CC^{d-1})$.

It is now sufficient to show that, for any $\lambda>1$, $HC(C_{\phi_{\lambda,I}})=HC(C_{\phi_{e,I}})$. This is also due to a semigroup argument! Indeed, define, for $a>0$, $T_a=C_{\phi_{e^a,I}}$. Then $(T_a)_{a>0}$ is a strongly continuous semigroup hence by the Conejero-M\"uller-Peris theorem, $HC(T_a)=HC(T_1)$ for any $a>0$.
\end{proof}
The previous result and Theorem \ref{THMHEISENBERG} suggest the following natural question.
\begin{question}
Let $\mathcal A$ the set of all automorphisms of $\hd$ with $+\infty$ as attractive fixed point. Is $\bigcap_{\phi\in\mathcal A}HC(C_\phi)$ nonempty?
\end{question}

\providecommand{\bysame}{\leavevmode\hbox to3em{\hrulefill}\thinspace}
\providecommand{\MR}{\relax\ifhmode\unskip\space\fi MR }
\providecommand{\MRhref}[2]{%
  \href{http://www.ams.org/mathscinet-getitem?mr=#1}{#2}
}
\providecommand{\href}[2]{#2}

\end{document}